\documentclass[12pt]{amsart}
\usepackage{amsmath,amssymb,amsfonts,amsthm,amsopn}
\usepackage{latexsym,graphicx}
\usepackage{xcolor}
\usepackage{color, colortbl}
\usepackage{pb-diagram}
\usepackage[title]{appendix}
\usepackage{tikz-cd}
\usepackage{tikz}
\usepackage{enumerate,hyperref}  
\usepackage{mathtools}
\usepackage{algorithm}
\usepackage{algpseudocode}

\mathtoolsset{showonlyrefs}

\calclayout

\usepackage{geometry}
\usepackage{multirow}

\newcommand{\ft}{Fourier transform}

\newtheorem{theorem}{Theorem}[section]
\newtheorem{lemma}[theorem]{Lemma}
\newtheorem{corollary}[theorem]{Corollary}
\newtheorem{proposition}[theorem]{Proposition}
\newtheorem{definition}[theorem]{Definition}

\newtheorem{remark}[theorem]{Remark}
\theoremstyle{definition}

\newcommand{\beqa}{\begin{eqnarray*}}
\newcommand{\eeqa}{\end{eqnarray*}}

\DeclareMathOperator*{\Sp}{Sp}

\DeclareMathOperator*{\Mp}{Mp}

\DeclareMathOperator*{\diag}{diag}
\DeclareMathOperator*{\GL}{GL}

\DeclareMathOperator*{\sgn}{sgn}

\newcommand{\field}[1]{\mathbb{#1}}
\newcommand{\bR}{\field{R}}        
\newcommand{\bN}{\field{N}}        
\newcommand{\bZ}{\field{Z}}        
\def\la{\lambda}

\def\eps{\varepsilon}

\def\cF{\mathcal{F}}              
\def\cS{\mathcal{S}}
\def\cD{\mathcal{D}}

\def\cB{\mathcal{B}}

\def\cA{\mathcal{A}}
\def\cJ{\mathcal{J}}
\def\cI{\mathcal{I}}
\def\cC{\mathcal{C}}

\def\cX{\mathcal{X}}
\def\cZ{\mathcal{Z}}
\def\SR{\mathrm{U}(2d,\bR)}

\def\wh{\widehat}

\def\rd{\bR^d}

\def\rdd{{\bR^{2d}}}

\def\lrd{L^2(\rd)}

\def\R{\right)}

\def\<{\left<}
\def\>{\right>}

\def\mv1{M_v^1}

\def\mn{(m,n)}
\def\mn'{(m',n')}

\newcommand{\abs}[1]{\lvert#1\rvert}
\newcommand{\norm}[1]{\lVert#1\rVert}

\def\i{\infty}

\def\R{\mathbb{R}}
\def\Ren{\mathbb{R}^d}

\def\f{\varphi}

\def\Sn2{S_{2}(L^{2}(\Ren))}
\def\S1{S_{1}(L^{2}(\Ren))}
\def\sig00{\sigma_{0,0}}

\def\la{\langle}
\def\ra{\rangle}

\def\w{\mathrm{w}}
\newcommand*\dd[1]{\mathop{}\!\mathrm{d}#1}

\begin{document}
\begin{abstract} Motivated by the phase space analysis of Schr\"odinger evolution operators, in this paper we investigate how metaplectic operators are approximately diagonalized along the corresponding symplectic flows by exponentially localized Gabor wave packets. Quantitative bounds for the matrix coefficients arising in the Gabor wave packet decomposition of such operators are established, revealing precise exponential decay rates together with subtler dispersive and spreading phenomena. To this aim, we present several novel results concerning the time-frequency analysis of functions with controlled Gelfand-Shilov regularity, which are of independent interest. 

As a byproduct, we generalize Vemuri's Gaussian confinement results for the solutions of the quantum harmonic oscillator in two respects, namely by encompassing general exponential decay rates as well as arbitrary quadratic Schr\"odinger propagators. In particular, we extensively discuss some prominent models such as the harmonic oscillator, the free particle in a constant magnetic field and fractional Fourier transforms.
\end{abstract}

\title[]{Sparse Gabor representations of metaplectic operators: controlled exponential decay and Schr\"odinger confinement}

\author[E. Cordero]{Elena Cordero}
\address{Universit\`a di Torino, Dipartimento di Matematica, via Carlo Alberto 10, 10123 Torino, Italy}
\email{elena.cordero@unito.it}

\author[G. Giacchi]{Gianluca Giacchi}
\address{Universit\`a della Svizzera Italiana, Faculty of Informatics, Via la Santa 1, 6962 Lugano, Switzerland}
\email{gianluca.giacchi@usi.ch}

\author[E. Pucci]{Edoardo Pucci}
\address{Universit\`a di Torino, Dipartimento di Matematica, via Carlo Alberto 10, 10123 Torino, Italy}
\email{edoardo.pucci@unito.it}

\author[S. I. Trapasso]{S. Ivan Trapasso}
\address{Politecnico di Torino, Dipartimento di Scienze Matematiche ``G. L. Lagrange'', corso Duca degli Abruzzi 24, 10129 Torino, Italy}
\email{salvatoreivan.trapasso@polito.it}

\thanks{}
\subjclass[2020]{42B35, 81S30, 35S10, 43A65, 42B37}
\keywords{Gelfand-Shilov spaces, metaplectic operators, time-frequency analysis, Gabor wave packets, Schr\"odinger equation, confinement, Euler decomposition, harmonic oscillator}
\maketitle

\section{Introduction}

The analysis of Schr\"odinger equations has largely benefited from techniques of time-frequency analysis in the last decades --- the recent monograph \cite{Elenabook} and the references therein could serve as a fair point of departure in this connection. Several non-trivial results, ranging from the study of well-posedness in low-regularity spaces to novel facets of the uncertainty principle, have been proved using techniques revolving around the paradigm of decomposition into Gabor wave packets \cite{folland89,grochenig}. Special attention has been reserved to the class of \textit{metaplectic operators} for several reasons (cf.\ Section \ref{subsec:metaplectic} for further details), including the fact that they model the Schr\"odinger evolution operators associated with quadratic differential operators, the most notable example being the quantum harmonic oscillator. To be more precise,
consider the problem (see, e.g., \cite[Sec.\ 4.3.3]{NT} or \cite[Sec.\ 15]{Gos11} with $\hbar=1/(2\pi)$)
\begin{equation}\label{SwQh}
i \frac{1}{2\pi}\frac{\partial u}{\partial t}(t,x) = H^\w u(t,x),
\end{equation}where $H^\w$ denotes the Weyl quantization of the real quadratic form $H\colon \rdd \to \bR$, that is
\begin{equation}
H^\w f(x) = \int_{\mathbb{R}^{2d}} e^{2\pi i (x-y) \cdot \xi} H\left(\frac{x + y}{2}, \xi\right) f(y) \dd{y} \dd{\xi}.
\end{equation} The harmonic oscillator $\displaystyle H=-\frac{\Delta}{8\pi^2}  +\frac{|x|^2}{2}$ corresponds to the quantization of the symbol   
$$H(x,\xi)= \frac{x^2+\xi^2}{2},\quad (x,\xi)\in\rdd.$$
The solution of the corresponding Cauchy problem with initial datum $u(0,x) = u_0(x)$ can be expressed using the metaplectic operator
\begin{equation} \label{eq-intro-metap}
\widehat{S}_t^H \coloneqq e^{-2\pi it H^\w},
\end{equation}
which evolves the initial state over time, that is $u(t,x)=\wh{S}_t^H u_0(x)$. 

The Hermite operator is actually a central subject of investigations in harmonic analysis, mostly in light of its distinctive spectral structure and related consequences. Among the several contributions on the matter, we quote  Vemuri's work \cite{vemuri} that has recently attracted renewed attention in the community \cite{kulikov,radchenko} --- see also \cite{MR1993414,cassano, cowling,garg} for related problems. In short, the problem concerns the Gaussian decay rates of solutions of the harmonic oscillator. For $\eps>0$ and $1 \le p \le \infty$, the space 
\[ E^p_\eps(\rd) \coloneq \{ f \in L^2(\rd) : \norm{f(x)e^{\eps|x|^2}}_{L^p_x}<\infty \,\text{ and }\, \norm{\hat f(\xi)e^{\eps|\xi|^2}}_{L^p_\xi}< \infty \} \] happens to be particularly interesting, as many uncertainty principles of Hardy or Cowling-Price type \cite{MR729369} can be readily characterized --- for instance, if $f \in E^\infty_\eps(\rd)$ we have the following trichotomy: if $\eps>\pi$ then $f \equiv 0$, while if $\eps=\pi$ then $f$ is proportional to a Gaussian function $e^{-\pi |x|^2}$, and if $\eps<\pi$ then $E^\infty_\eps(\rd)$ is an infinite-dimensional space including all the Hermite functions. 

Let us focus on the case where $p=\infty$. In dimension $d=1$, Vemuri proved fine properties of a general function $f$ belonging to $E^\infty_\eps(\bR)$ in the subcritical regime $\eps<\pi$, showing that the coefficients of the Hermite expansion of $f$ satisfy a precise exponential decay. Interestingly enough, this result has been used to prove a localization result for the solutions of the Hermite operator: given an initial datum $\psi_0 \in E^\infty_{\tanh(2\alpha)}(\bR)$, the solution $\psi_t = e^{-2\pi i t H^\w}\psi_0$ belongs to $E^\infty_{\tanh(\beta)}(\bR)$ for all $t>0$ and $\beta<\alpha$ --- the endpoint conjecture $\beta=\alpha$ was proved only recently in \cite{kulikov,radchenko}. 

Vemuri refers to this result as a manifestation of \textit{confinement} of solutions of the harmonic oscillator. In fact, from a heuristic point of view, the essential phase-space content of a function $f \in E^\infty_\eps(\bR)$ is sharply localized (i.e., with Gaussian decay) in a square box of side $\sim 1/\eps$. The problem just outlined strongly resonates with some recent trends in the area of Gabor wave packet analysis, as we now highlight. 

\subsection{Gelfand-Shilov spaces with fixed rates} 

First, we emphasize that the spaces $E^\infty_\eps(\rd)$ under our attention are special cases of the well known \textit{Gelfand-Shilov spaces}, introduced in \cite{gelfand1955,GS3,GS2} in the context of the analysis of certain parabolic initial value problems. Since their appearance, they have been successfully employed to study the well-posedness and properties of PDEs' solutions in various settings \cite{Cappiello2010b,Cappiello2010a,LernerEtAl2014}, also using time-frequency techniques \cite{BoitietalJFA2020,fio3,fio2015,spreading}. In fact, certain Gelfand-Shilov spaces can be characterized by time-frequency representations and described in terms of modulation spaces, with applications to pseudodifferential calculus \cite{GZ,Pilipovic2020,PTT2024,T2,T1,Toft2004a,Toft2004b,ToftGS,Toft2013,Toft2015,Toft2017b}. 

To be more concrete, recall that the Gelfand-Shilov space $\cS^r_s (\rd)$ ,  $r,s\ge 1/2$, contains all the smooth functions $f\in C^\infty (\rd)$ for which there exist constants $h>0$, $k>0$ such that 
\begin{equation}\label{GS}
 \norm{f(x)e^{h|x|^{1/r}}}_{L^\infty_x}<\i\,\,\mbox{ and}\,\,\,\,\norm{\hat f(\xi)e^{k|\xi|^{1/s}}}_{L^\infty_\xi}<\i,
\end{equation} 
see \cite{medit,GZ,T2,T1}. We focus here on the Fourier-invariant spaces with $r=s$, and the key observation for our purposes is that the constants $h,k$ are not generally related. The link with Vemuri spaces (actually corresponding to $s=1/2$) comes when we restrict to the following subclasses of Gelfand-Shilov spaces.
\begin{definition}\label{E1}
For $s\geq1/2$, $\eps>0$, the $\eps$-Gelfand Shilov class $\cS^s_{s,\eps}(\rd)$, is the space of functions $f\in C^\infty (\rd)$
such that 
\begin{equation}\label{Elenah}
 \norm{f(x)e^{\eps|x|^{1/s}}}_{L^\infty_x}<\i\,\,\mbox{ and}\,\,\,\,\norm{\hat f(\xi)e^{\eps|\xi|^{1/s}}}_{L^\infty_\xi}<\i.
\end{equation}
\end{definition}
It is clear from the characterization in \eqref{GS} that $$\cS^s_{s,\eps}(\rd)\subset \cS^s_{s}(\rd)\,\,\mbox{ and}\quad\,\bigcup_{\eps>0}\cS^s_{s,\eps}(\rd)=\cS^s_{s}(\rd).$$Similar spaces of functions have been investigated recently in connection with the exponential decay of the Hermite coefficients, see for instance \cite{neyt,toft_23}.

The first goal of our study is to understand the time-frequency concentration of functions in these subclasses. To this aim, recall that if $g \in \cS(\rd)$ is a non-trivial Schwartz function, the \textit{wave packet} $\pi(z)g(y) = e^{2\pi i \xi y}g(y-x)$ arising from a joint translation and modulation of $g$ is then essentially localized in phase near $z=(x,\xi) \in \rdd$. A standard phase space representation of a tempered distribution $f\in\cS'(\rd)$ is given by the \textit{short-time Fourier transform (STFT)}, which ultimately amounts to compute coefficients of a continuous decomposition of $f$ into Gabor wave packets:
\begin{equation}\label{FTdef}
	V_gf (x,\xi)\coloneqq \langle f, \pi(x,\xi) g\rangle=\int_{\rd}e^{-2\pi iy \cdot \xi }
	f(y)\, {\overline {g(y-x)}} \dd{y}.
\end{equation}

It was shown by  Gr\"ochenig and Zimmermann in \cite{GZ} that, if $f,g\in\cS^s_{s}(\rd)$, then $V_gf\in \cS^s_{s}(\rdd)$. We refine the mentioned result as follows.

\begin{theorem}\label{stftdecay}
Let $f, g \in \cS^s_{s,\eps}(\rd)$ with $s \ge \frac{1}{2}$ and $\eps>0$. Then $V_g f$ belongs to $\cS^s_{s,\delta}(\rdd),$  for every $\delta>0$ satisfying
\begin{equation}\label{delta-stft}
\begin{cases}
    \delta=\eps/4, & s=1/2 \\
    \delta <    2^{1-\frac{3}{2s}} \, \eps, & 1/2 < s < 1\\
    \delta < 2^{-\frac{1}{2s}}\, \eps, & s \ge 1. 
\end{cases}
\end{equation}
\end{theorem} Note that the transition from the convex to concave regime at $s=1$ in the threshold bounds is continuous, and there is reason to believe that the latter are optimal. The loss that prevents one from achieving them with $\delta$ (except for the Gaussian case) originates from the current proof strategy, relying on careful bounds for convolutions of exponential functions --- see Section \ref{sec-tech} for the details. We emphasize that one can easily obtain similar results for the Wigner distribution and the ambiguity function, as well as other time-frequency representations, as detailed in Section \ref{sec-tfrep} below.

On the other hand, a partial converse of Theorem \ref{stftdecay} can be obtained if we fix the window $g$ and by asking for a stronger decay on the STFT.
\begin{theorem}\label{viceversaSTFT}
    Let $s\geq 1/2$ and $g\in \cS^s_{s,\eps}(\rd)\setminus\{0\}$ be fixed. If there exists  $C>0$ such that
    \begin{equation}\label{1.3}
        |V_gf(z)| \le C e^{-2^{1-\frac1s}\eps|z|^{\frac1s}}, \quad z\in \rdd,
    \end{equation} 
    then $f\in\cS^s_{s,\eps}(\rd)$.
\end{theorem} A more extensive analysis of the Gaussian scenario is carried out in Section \ref{sec-gau} below.

\subsection{Exponential confinement for metaplectic operators} Although the exponential constraints entailed by the Gelfand-Shilov regularity leave essentially no room for meaningful uncertainty principles when $s>1/2$, Vemuri's confinement results can be generalized, beyond the Gaussian case, to more general metaplectic operators. To this aim, we first recall that in order to study how a linear operator $T \colon \cS(\rd) \to \cS'(\rd)$ interacts with wave packets, we introduce the \textit{Gabor matrix} of $T$ with respect to atoms $g,\gamma \in \cS(\rd)\setminus\{0\}$, that is 
\[
\langle T \pi(z)g, \pi(w)\gamma \rangle, \quad z, w\in \mathbb{R}^{2d}.
\] Assuming $\norm{g}_{L^2}=1$ for convenience, the inversion of the STFT (cf.\ \cite{grochenig}) allows one to lift the analysis of $T$ to the phase space level via the following identity:
\begin{equation}\label{eq-intro-lens}
	V_\gamma(T f)(w) = \int_{\rdd} \langle T \pi(z)g, \pi(w)\gamma \rangle V_g f(z) \dd{z}, \quad w\in\rdd.
\end{equation} The Gabor matrix thus coincides with the kernel of an integral operator associated with $T$, and plays a crucial role in continuity results and sparse representations --- in line with other discrete kernels build up with frames, see for instance \cite{cddy,candes,fio3,Elenabook,guo-labate}. 

A sparse Gabor matrix is a highly desirable feature for an operator $T$: just like a sparse matrix offers computational advantages for finite-dimensional operators, a sparse Gabor matrix provides a parsimonious representation due to the limited number of non-negligible entries. While standard analytical approaches often give a general, qualitative understanding of how matrix entries decay far from the diagonal, these methods are generally not strong enough for theoretical and practical applications that require a precise quantification of the decay rate. In this connection, several results in the recent literature concern the study of the Gabor matrix structure, cf., e.g., \cite{fio3,fio2015,Gaborevol,spreading,Ivan}. In particular, in \cite{spreading} the authors obtained quantitative bounds for the entries of the Gabor matrix of metaplectic operators $\wh{S}$. To be definite, in the case \eqref{eq-intro-metap} outlined above, for every $N \in \bN$ and $z,w \in \rdd$, 
\begin{equation}
	\abs{\langle \wh{S}_t^H \pi(z)g, \pi(w) \gamma \rangle} \le C (\sigma_1(t)\cdots \sigma_d(t))^{-1/2} (1+\abs{M_t(w-S_tz)})^{-N},
\end{equation} for a suitable $C>0$ independent of $t$.  The following phenomena are highlighted: 
\begin{itemize}
    \item Gabor wave packets are approximately evolved along the graph of the symplectic flow $t \mapsto S_t$, with a decay rate away from the latter consistent with the Schwartz regularity of the atoms $g,\gamma$. 
    \item The occurrence of the largest singular values $\sigma_1(t) \ge \cdots \ge \sigma_d(t) \ge 1$ of $S_t$ are tied to dispersion \cite{cauli}.
    \item The matrix $M_t \in \bR^{2d,2d}$ accounts for the expected quantum spreading phenomenon incurred by wave packets, resulting in a phase space envelope along the classical trajectory.
\end{itemize}
 
If atoms $g,\gamma$ are taken in Gelfand-Shilov classes, one expects to move from a superpolynomial regime to general exponential decay rates for the Gabor matrix, promoting higher degrees of sparsity. The Gabor matrix decay with Gelfand-Shilov window $g\in \cS^{s}_{s}(\rd)$ was already studied in \cite[Section 5]{fio2015}, leading to bounds of the form
	\begin{equation}\label{qualitativeDecay} |\langle \wh S
		\pi(z)g,\pi(w)g\rangle|\leq Ce^{-a|w-S(z)|^{1/s}},\quad
		z,w\in\rdd, 
	\end{equation} 
    for some constants  $C,a>0$.  We underline that  the constant $a>0$ is unknown and, in this respect, the decay in \eqref{qualitativeDecay} is only qualitative. It is therefore natural to restrict our attention to $\eps$-Gelfand–Shilov spaces, which allows for precise control over the size parameter $a$, and to refine the estimate so that the more subtle effects revealed in \eqref{eq-intro-metap}  become apparent. 

To state our result in the most general form, we briefly recall  that the singular value decompositions of a symplectic matrix $S \in \Sp(d,\bR)$ can be taken in a very special form (often referred to as \textit{Euler decomposition} in the literature) --- see Proposition \ref{euler dec} below for more details. First, the singular values appear in reciprocal pairs of positive numbers, so we arrange them in such a way that $\sigma_1 \ge \ldots \ge \sigma_d \ge 1$ are the largest ones and introduce the diagonal matrix $\Sigma = \diag(\sigma_1, \ldots, \sigma_d)$. We also need to consider the associated diagonal matrices
\begin{equation}\label{svd}\quad D=\begin{pmatrix} \Sigma & O \\ O & \Sigma^{-1} \end{pmatrix}, \quad D'=\begin{pmatrix} \Sigma^{-1} & O \\ O & I \end{pmatrix}, \quad D''=\begin{pmatrix} I & O \\ O & \Sigma^{-1} \end{pmatrix}.
\end{equation}
Moreover, there exist non-unique \textit{symplectic rotations} (i.e., orthogonal and symplectic matrices) $U,V \in \bR^{2d,2d}$ such that $S = U^\top D V$ (see \cite{serafini} and Subsection \ref{subsec:metaplectic} below). Any such decomposition is then characterized by the triple $(U, V, \Sigma)$. Henceforth, given a matrix $S \in \Sp(d,\bR)$, we will denote by $(U, V, \Sigma)$ a corresponding Euler decomposition, and by $D, D', D''$ the associated matrices as defined above.

We are now ready to present a refined pointwise estimate for the entries of the Gabor matrix of a metaplectic operator, computed using an $\eps$-Gelfand-Shilov atom.  
\begin{theorem}\label{teoe00}
Consider $s\geq 1/2$ and $g,\gamma\in\cS^s_{s,\eps}(\rd)$. Then, there exists $C>0$ such that, for every $\widehat S\in Mp(d,\bR)$ with projection $S\in \Sp(d,\bR)$, and every Euler decomposition $(U,V,\Sigma)$ of $S$, we have
\begin{equation}\label{spreading}
    |\la \wh S\pi(z)g,\pi(w)\gamma\ra| \le C \det(\Sigma)^{-1/2}e^{-\delta(s,d)|D'U(w-Sz)|^{1/s}},\quad w,z\in\rdd,
\end{equation} where $D'$ is defined in \eqref{svd} and $\delta(s,d)>0$ is given by
\begin{equation}\label{eq-delta}
    \begin{cases} \delta(s,d) = d^{-1} \eps, & s=1/2 \\ 
   \delta(s,d) < d^{-1/(2s)} 2^{2-1/s} \eps & 1/2 < s < 1 
   \\ \delta(s,d)  <
    d^{-1/2} 2^{-7/2+1/(2s)} \eps, & s\ge 1.
    \end{cases}
\end{equation}
\end{theorem}
In particular, for $z=0$ we recapture the STFT of $\wh{S}$ and thus obtain information about the related $\eps$-Gelfand-Shilov class. 
\begin{corollary}\label{cor0e}
Consider $s\geq 1/2$ and $f,g\in\cS^s_{s,\eps}(\rd)$. Then, there exists $C>0$ such that, for every $\widehat S\in Mp(d,\bR)$,
\begin{equation}
    |V_g (\wh Sf)(z)| \le C \det(\Sigma)^{-1/2}e^{-\delta(s,d) \sigma_{\min}^{1/s}|z|^{1/s}},\quad z\in\rdd,
\end{equation} where $0<\sigma_{\min}\leq 1$ 
is the  smallest singular value of $S$ and $\delta(s,d)>0$ is as in \eqref{eq-delta}. 
\end{corollary}
Note that this result, combined with Theorem \ref{viceversaSTFT}, leads to a fully fledged confinement result in the spirit of Vemuri: 
\begin{corollary}
    Let $\wh{S}_t^H=e^{-2\pi itH^\w}$ be the Schr\"odinger propagator associated with a quadratic Hamiltonian $H$, as detailed above. For all $s\ge 1/2$ and $\eps>0$, 
    \[ \text{ if } \quad  u_0 \in \cS^s_{s,\eps}(\rd) \quad \text{ then } \quad u_t=\wh{S}_t^Hu_0 \in \cS^s_{s,\eps(t)},\] where $\eps(t)=2^{-1+1/s}\delta(s,d)\sigma_{\min}(t)^{1/s}$ and $\delta(s,d)>0$ is as in \eqref{spreading}.
\end{corollary}

It is of course not surprising that in the case of the harmonic oscillator with $s=1/2$ our result ($\eps(t)=\eps/4d$) is not as strong as the one obtained by Vemuri and others, who in fact exploited tools that better adapt to the fine structure of the solutions in that case (e.g., exponential decay of the Hermite coefficients). Nonetheless, our result has the merit to generalize the confinement principle beyond the Gaussian regularity and to every metaplectic operator, hence resulting in a significantly broader scope. Refined additional results in this spirit are discussed in Section \ref{sec-metdecay} with full details. 

For the benefit of the reader, in Section \ref{sec:appl} below we collect some results concerning explicit Euler decompositions for a number of interesting scenarios --- including the free particle (possibly in presence of a constant magnetic field) and the anisotropic harmonic oscillator. {A detailed analysis of generators of the symplectic group and related operators (e.g., fractional Fourier transforms) is performed as well.}
\medskip

\noindent
The paper is organized as follows. In Section \ref{sec-prel}, we introduce the necessary background on symplectic matrices, metaplectic operators, and time-frequency representations. In Section \ref{sec-tech}, we collect technical lemmas concerning exponential integrals and Gelfand-Shilov regularity. The main results (Theorems \ref{stftdecay}, \ref{viceversaSTFT}, \ref{teoe00}, and related corollaries) are then proved in Section \ref{sec-main}, while Schr\"odinger confinement properties in Gelfand-Shilov spaces under the action of metaplectic operators are discussed in Section \ref{sec-metdecay}. To conclude, in Section \ref{sec:appl} we provide explicit Euler decompositions for several models of interest.

\section{Preliminaries}\label{sec-prel}

\subsection{Notation} 
We denote by $xy=x\cdot y$ the standard inner product in $\rd$. The Fourier transform of a function $f\in \cS(\rd)$ is
\begin{equation}
    \hat f(\xi)=\int_{\rd}e^{-2\pi i\xi x} f(x) \dd x, \qquad \xi\in\rd.
\end{equation}
For $f,g\in L^2(\rd)$, we denote by $\la f,g\ra$ the sesquilinear inner product in $L^2(\rd)$. It extends uniquely to a duality pairing $\la\cdot,\cdot\ra$ on $\cS'(\rd)\times\cS(\rd)$, where $\cS(\rd)$ denotes the Schwartz class of smooth, rapidly decaying functions, conjugate-linear in the second component. The Fourier transform of a tempered distribution $f\in\cS'(\rd)$ is defined accordingly by
\begin{equation}
    \la \hat f,\hat \f\ra =\la f,\f\ra, \qquad \f\in\cS(\rd).
\end{equation}
For $d\times d$ real matrices $A,B\in \bR^{d\times d}$, we write
\begin{equation}
    A\oplus B=\begin{pmatrix}
        A & O_d\\
        O_d & B
    \end{pmatrix},
\end{equation}
where $O_d$ denotes the $d\times d$ null matrix. We also denote by $I_d$ the $d\times d$ identity matrix. When the dimension $d$ is understood, we omit the subscript and write $O$ and $I$. If $(\lambda_j)_{j=1}^d\subseteq\bR$, we denote by $\diag(\lambda_j)$ the diagonal matrix with ordered diagonal entries $\lambda_1,\ldots,\lambda_d$. 
We denote by $\mbox{U}(d,\bR)$ the set of $d\times d$ orthogonal matrices, whereas $\GL(d,\bR)$ denotes the space of invertible $d\times d$ matrices.

Finally, we write $f\lesssim g$ when the inequality $f\leq Cg$ holds with a constant $C>0$ independent of $f,g$.

\subsection{Symplectic matrices and metaplectic operators} \label{subsec:metaplectic}

The symplectic group $\Sp(d, \mathbb{R})$ consists of all real $2d \times 2d$ matrices $S$ satisfying the condition
\begin{equation} \label{fundIdSymp}
	S^\top  J S = J,
\end{equation}
where the canonical symplectic matrix $J$ is 
\begin{equation} \label{defJ}
	J = \begin{pmatrix}
		O_{d} & I_{d} \\
		 -I_{d} & O_{d}
	\end{pmatrix}.
\end{equation}

\subsubsection*{Euler decomposition} We will need the following result on a SVD-like decomposition of symplectic matrices, also known as the \textit{Euler decomposition} --- \cite[Appendix B.2]{serafini} for additional details. 
\begin{proposition} \label{euler dec}
	For any $S \in \Sp(d,\bR)$ there exist $U,V \in \SR$ which satisfy \[ S=U^\top DV, \quad D = \Sigma \oplus \Sigma^{-1},\]
	with $\Sigma=\diag(\sigma_1,\ldots,\sigma_d)$ and $\sigma_1 \ge \ldots \ge \sigma_d \ge \sigma_d^{-1} \ge \ldots \ge \sigma_1^{-1}=\sigma_{\min}$ are the singular values of $S$. 
	\end{proposition}
    Euler decompositions of the generators of $\Sp(d,\bR)$ are computed in Section \ref{subsec:metaplecticGens}.
    
The \emph{metaplectic group} \( \Mp(d,\mathbb{R}) \) is a double cover of the symplectic group \( \Sp(d,\mathbb{R}) \) \cite{Gos11,folland89}. This relationship is described by the group homomorphism
\begin{equation} \label{piMp}
	\pi^{\Mp} : \Mp(d,\mathbb{R}) \rightarrow \Sp(d,\mathbb{R}),
\end{equation}
whose kernel consists of the identity and its negative in the space \( L^2(\mathbb{R}^d) \), i.e., \( \ker(\pi^{\Mp}) = \{ \pm \mathrm{id}_{L^2} \} \). In the context of this work, any operator \( \wh{S} \in \Mp(d,\mathbb{R}) \) will be associated with the unique matrix \( S \in \Sp(d,\mathbb{R}) \) such that \( \pi^{\Mp}(\wh{S}) = S \). Moreover, we have:

\begin{proposition}[{\cite[Proposition 4.27]{folland89}}] \label{Folland427}
	Every metaplectic operator \( \wh{S} \in \Mp(d,\mathbb{R}) \) induces an isomorphism on the Schwartz space \( \mathcal{S}(\mathbb{R}^d) \), and this mapping extends to an isomorphism on the space of tempered distributions \( \mathcal{S}'(\mathbb{R}^d) \).
\end{proposition}

\subsection{Time-frequency representations}\label{sec-tfrep} Consider a distribution $f\in\cS '(\rd)$
and a function $g\in\cS(\rd)\setminus\{0\}$ (the so-called
{\it window}).
The short-time Fourier transform (STFT) of $f$ with respect to $g$ is defined in \eqref{FTdef}.
 The STFT is well-defined whenever the bracket $\langle \cdot , \cdot \rangle$ makes sense for
dual pairs of function or (ultra-)distribution spaces, in particular for $f\in
\cS ' (\rd )$ and $g\in \cS (\rd )$, $f,g\in\lrd$, or $f\in
(\cS^s_r) ' (\rd )$ and $g\in \cS^s_r (\rd )$ (see \cite{grochenig,GZ,T2} for full
details). \par

In what follows we list the basic properties of the STFT we are going to use in the sequel, see, e.g., \cite{Elenabook,grochenig}.
\begin{proposition}\label{STFTprop}
	For $f,g\in\cS(\rd)$, the STFT enjoys the following properties:
    \begin{itemize}
	\item[(i)] STFT of the Fourier transforms:
	\begin{equation}\label{ed1}
		V_{g}f(x,\xi)=e^{-2\pi ix\cdot\xi}V_{\hat g}\hat f(\xi,-x),\quad x,\xi\in\rd;
	\end{equation}
	\item[(ii)] Fourier transform of the STFT:
	 \begin{equation}\label{FTSTFT}
		\widehat{V_gf}(\omega,\eta)=e^{2\pi i y\cdot \eta}f(-\eta)\overline{\hat g(\omega)},\quad \eta,\omega\in\rd.
	\end{equation}
    \end{itemize}
\end{proposition}

\subsubsection*{Wigner distribution} For $f,g\in L^2(\rd)$, the {\em cross-Wigner distribution} is defined by
\begin{equation}
	W(f,g)(x,\xi)=\int_{\rd}f(x+t/2)\overline{g(x-t/2)}e^{-2\pi i\xi t}\dd t, \qquad x,\xi\in\rd. 
\end{equation} Again, the definition extends to (ultra-)tempered distributions with formal meaning of the integral --- we refer to \cite{Elenabook,grochenig} for the details. We write $Wf=W(f,f)$. The Wigner distribution and the short-time Fourier transform are related by
\begin{equation}\label{Wig-stft}
	W(f,g)(x,\xi)=2^{d}e^{4\pi ix\xi}V_{\mathcal{I}g}f(2x,2\xi),\quad x,\xi\in\rd,
\end{equation}
where $\mathcal{I}g(t)=g(-t)$ is the flip operator.
Moreover, 
\begin{equation}\label{trasf-Wigner}
    \cF W(f,g)(y,\eta)=e^{-\pi i y\eta}V_gf(-\eta,y).
\end{equation}
We recall the following important property on the marginals of the Wigner distribution: 
\begin{proposition}\label{Marginals}
    If $f,g,\hat f,\hat g\in L^1(\rd)$, then:
    \begin{itemize}
    \item[(i)] $
        \int_{\rd}W(f,g)(x,\xi)\dd{\xi}=f(x)\overline{g(x)},\quad \forall x\in\rd.$
        \item[(ii)]$
            \int_{\rd}W(f,g)(x,\xi)\dd{x}=\hat f(\xi)\overline{\hat g(\xi)},\quad \forall \xi\in\rd.$
    \end{itemize}
\end{proposition}
We recall the following properties of the Wigner distribution:
\begin{itemize}
    \item \textit{Covariance under time-frequency shifts:} for all $f,g\in L^2(\rd)$ and $z,w\in\rdd$,
\begin{equation}\label{eq-wigcov}
    W(\pi(z)f,\pi(z)g)(w)=W(f,g)(w-z). 
\end{equation}
\item \textit{Symplectic covariance:} for all $f,g\in L^2(\rd)$ and $z\in\rdd$,
\begin{equation}\label{eq-sympcovwig}
    W(\wh S f ,\wh S g)(z) = W(f,g)(S^{-1}z). 
\end{equation}
\end{itemize}

\section{Technical lemmas}\label{sec-tech}

We will make use of the following properties for sequences of complex numbers, see for instance \cite{BObook}.
\begin{lemma} \label{lemma3.1}
Let $a = (a_k)_{k \in \mathbb{N}}$ be a sequence of complex numbers, $n \geq 1$ a fixed integer, and $p \geq 0$.
\begin{itemize}
	\item[(a)] If $1 \leq p < \infty$, then
	\[
	\left( \sum_{k=1}^n |a_k| \right)^p \leq n^{p-1} \sum_{k=1}^n |a_k|^p \quad \text{and} \quad \sum_{k=1}^n |a_k|^p \leq \left( \sum_{k=1}^n |a_k| \right)^p.
	\]
	
	\item[(b)] If $0 \leq p < 1$, then
	\[
	\sum_{k=1}^n |a_k|^p \leq n^{-p+1} \left( \sum_{k=1}^n |a_k| \right)^p \quad \text{and} \quad \left( \sum_{k=1}^n |a_k| \right)^p \leq \sum_{k=1}^n |a_k|^p.
	\]
	\end{itemize}
\end{lemma}

Our next goal is to review some technical convolution lemmas proved in \cite{spreading} in light of the currently assumed Gelfand-Shilov regularity. First, given $\eps>0$ and $0<s<\infty$, we define
\begin{equation}\label{cs}
	C(\eps,s)\coloneqq \int_{\R}e^{-\frac{\eps}{2}|u|^{\frac{1}{s}}} \dd{u}.
	\end{equation}

\begin{lemma}\label{lemmatec1}
Let $s\geq 1/2$, $a,b\geq 0$, $\sigma\geq 1$, $v\in\R$ and $\eps>0$. Then \begin{equation}\label{expCordero1}
   \int_{\R} e^{-\eps(a+|\sigma^{-1}u-v|)^{\frac{1}{s}}}e^{-\eps(b+|u|)^{\frac{1}{s}}} \dd{u} \leq C(\eps,s) (e^{-\eps2^{-\frac{1}{s}}(a+|v|)^{\frac{1}{s}}} e^{-\frac{\eps}{2}b^{\frac{1}{s}}}+e^{-\eps2^{-(\frac{1}{s}+1)}(b+|v|)^{\frac{1}{s}}}e^{-\eps a^{\frac{1}{s}}}),
\end{equation}
where $C(\eps,s)$ is defined in \eqref{cs}.
\end{lemma}
\begin{proof}
    First we assume $|\sigma^{-1}u-v|\geq \frac{|v|}{2}$, so 
    \begin{equation}
        e^{-\eps(a+|\sigma^{-1}u-v|)^{\frac{1}{s}}}\leq e^{-\eps 2^{-\frac{1}{s}}(a+|v|)^{\frac{1}{s}}}
    \end{equation}
    and \begin{equation}
        \int_{\R} e^{-\eps(b+|u|)^{\frac{1}{s}}} \dd{u} \leq e^{-\frac{\eps}{2}b^{\frac{1}{s}}}\int_{\R}e^{-\frac{\eps}{2}(b+|u|)^{\frac{1}{s}}} \dd{u} =C(\eps,s) e^{-\frac{\eps}{2}b^{\frac{1}{s}}}.
    \end{equation}
    Hence \begin{equation}
         \int_{\R} e^{-\eps(a+|\sigma^{-1}u-v|)^{\frac{1}{s}}}e^{-\eps(b+|u|)^{\frac{1}{s}}} \dd{u} \leq C(\eps,s)e^{-\eps2^{-\frac{1}{s}}(a+|v|)^{\frac{1}{s}}} e^{-\frac{\eps}{2}b^{\frac{1}{s}}}.
    \end{equation}
 On the other hand, if $|\sigma^{-1}u-v|\leq\frac{|v|}{2}$, we have in particular that $|u|\geq\sigma\frac{|v|}{2}$ which implies:\begin{equation}
     e^{-\frac{\eps}{2}(b+|u|)^{\frac{1}{s}}}\leq  e^{-\frac{\eps}{2} 2^{-\frac{1}{s}}(b+\sigma|v|)^{\frac{1}{s}}}\leq e^{-\frac{\eps}{2} 2^{-\frac{1}{s}}(b+|v|)^{\frac{1}{s}}}.
 \end{equation}
 Now, by the trivial  inequality $a+|\sigma^{-1}u-v|\geq a$, we deduce that $e^{-\eps(a+|\sigma^{-1}u-v|)^{\frac{1}{s}}}\leq e^{-\eps a^{\frac{1}{s}}},$  and  obtain\begin{equation}
     \int_{\R}e^{-\frac{\eps}{2}(b+|u|)^{\frac{1}{s}}} e^{-\frac{\eps}{2}(b+|u|)^{\frac{1}{s}}} e^{-\eps(a+|\sigma^{-1}u-v|)^{\frac{1}{s}}} \dd{u} \leq C(\eps,s)e^{-\eps2^{-(\frac{1}{s}+1)}(b+|v|)^{\frac{1}{s}}}e^{-\eps a^{\frac{1}{s}}},
 \end{equation}
 as desired.
\end{proof}

In the case where $a=b=0$ we infer:
\begin{corollary}\label{core0}
	Let $s \geq 1/2$, $\sigma \geq 1$, $v \in \mathbb{R}$, and $\varepsilon > 0$. Then
	\[
	\int_{\mathbb{R}} e^{-\varepsilon |\sigma^{-1} u - v|^{1/s}} e^{-\varepsilon |u|^{1/s}} \,  \dd{u} 
	\leq C(\varepsilon, s)\left(
	e^{-\varepsilon 2^{-1/s} |v|^{1/s}} + e^{-\varepsilon 2^{-(1/s + 1)} |v|^{1/s}}
	\right),
	\]
which gives
	\[
	\int_{\mathbb{R}} e^{-\varepsilon |\sigma^{-1} u - v|^{1/s}} e^{-\varepsilon |u|^{1/s}} \,  \dd{u} 
	\leq 2 C(\varepsilon, s) e^{-\varepsilon 2^{-(1/s + 1)} |v|^{1/s}}.
	\]
\end{corollary}

\begin{lemma}\label{lemmatec2}
Under the  assumptions of Lemma \ref{lemmatec1} we have \begin{equation}\label{expCordero2}
  \int_{\R} e^{-\eps(a+|u-v|)^{\frac{1}{s}}}e^{-\eps(b+\sigma^{-1}|u|)^{\frac{1}{s}}} \dd{u} \leq C(\eps,s) (e^{-\eps2^{-(\frac{1}{s}+1)}(a+|v|)^{\frac{1}{s}}}e^{-\frac{\eps}{2}b^{\frac{1}{s}}}+e^{-\eps2^{-\frac{1}{s}}(b+\sigma^{-1}|v|)^{\frac{1}{s}}}e^{-\frac{\eps}{2}a^{\frac{1}{s}}}).
\end{equation}
\end{lemma}
\begin{proof}
   Let us define
   $$ I\coloneqq \int_{\R} e^{-\eps(a+|u-v|)^{\frac{1}{s}}}e^{-\eps(b+\sigma^{-1}|u|)^{\frac{1}{s}}} \dd{u} .$$ If $|u-v|\geq \frac{|v|}{2}$,\begin{equation}
        I=\int_{\R} e^{-\frac{\eps}{2}(a+|u-v|)^{\frac{1}{s}}}e^{-\frac{\eps}{2}(a+|u-v|)^{\frac{1}{s}}}e^{-\eps(b+\sigma^{-1}|u|)^{\frac{1}{s}}} \dd{u} \leq C(\eps,s)e^{-\eps2^{-(\frac{1}{s}+1)}(a+|v|)^{\frac{1}{s}}}e^{-\frac{\eps}{2}b^{\frac{1}{s}}},
    \end{equation}
   where we used the inequality $(b+\sigma^{-1}|u|)^{\frac{1}{s}}\geq b^{\frac{1}{s}}$, which implies  $e^{-\eps(b+\sigma^{-1}|u|)^{\frac{1}{s}}}\leq e^{-\eps b^{\frac{1}{s}}}$.\\
    If $|u-v|\leq \frac{|v|}{2}$, we have $|u|\geq \frac{|v|}{2}$ so that \begin{equation}
        I\leq C(\eps,s)e^{-\eps2^{-\frac{1}{s}}(b+\sigma^{-1}|v|)^{\frac{1}{s}}}e^{-\frac{\eps}{2}a^{\frac{1}{s}}}.
    \end{equation}
    This concludes the proof.
\end{proof}
\begin{corollary}\label{core1}
	Let $s \geq \tfrac{1}{2}$, $\sigma \geq 1$, $v \in \mathbb{R}$, and $\varepsilon > 0$. Then
	\[
	\int_{\mathbb{R}} e^{-\varepsilon |u - v|^{1/s}}\, e^{-\varepsilon |\sigma^{-1} u|^{1/s}}\,  \dd{u} 
	\leq C(\varepsilon, s) \left(
	e^{-\varepsilon 2^{-(1/s + 1)} |v|^{1/s}} + e^{-\varepsilon 2^{-1/s} |\sigma^{-1} v|^{1/s}}
	\right),
	\]
so that
		\[
	\int_{\mathbb{R}} e^{-\varepsilon |u - v|^{1/s}}\, e^{-\varepsilon |\sigma^{-1} u|^{1/s}}\,  \dd{u} 
	\leq2 C(\varepsilon, s) 
	e^{-\varepsilon 2^{-(1/s + 1)} |\sigma^{-1} v|^{1/s}}.
	\]
\end{corollary}

\begin{lemma}\label{lemmatec2d}
    Given $s\geq \frac{1}{2}$, for every $\eps>0$, $z\in\bR$,
    \begin{equation}\label{expCordero3}
        \int_{\rdd} e^{-\eps |D'u-v|^{\frac{1}{s}}}e^{-\eps|D''u|^{\frac{1}{s}}} \dd{u} \leq 	\, \left(\frac{2}{d}\right)^d  C(\eps,s)^{2d} \,
               	 e^{-\psi_s(\eps)|D''v|^{\frac{1}{s}}},
    \end{equation}
where $\psi_s({\eps})$ is defined by:
\begin{equation}\label{defpsi}
	\psi_s({\eps})=\begin{cases}
		\displaystyle\frac{4\eps}{(2d)^{\frac1{2s}}},\quad \frac12 \leq s<1\\
		\\
		\displaystyle\frac{\eps}{4(2d)^{\frac12}},\quad s\geq 1,
	\end{cases}
\end{equation}
and   $C(\eps,s)>0$ is defined \eqref{cs}.
\end{lemma}
\begin{proof}
    By using the equivalence between the $\ell^1$ and $\ell^2$ norms on $\rdd$, cf. Lemma \ref{lemma3.1}, $(a)$, we can write:
    \begin{multline}
        I  \coloneqq \int_{\rdd} e^{\displaystyle-\eps |D'u-v|^{\frac{1}{s}}}e^{-\displaystyle\eps|D''u|^{\frac{1}{s}}} \dd{u} \\
        \leq \int_{\rdd} \exp\Bigg(-\frac{\eps}{(2d)^{1/(2s)}}\left(\sum_{j=1}^d|\sigma_j^{-1}u_j-v_j|+\sum_{j=d+1}^{2d}|u_j-v_j|\right)^{\frac{1}{s}}\Bigg) \\ \times \exp\Bigg(-\frac{\eps}{(2d)^{1/(2s)}}\left( \displaystyle\sum_{j=1}^d|u_j|+\displaystyle\sum_{j=d+1}^{2d}|\sigma_{j-d}^{-1}u_j|\right)^{\frac{1}{s}}\Bigg)\dd{u} .
    \end{multline}
    Now, for $1/2\leq s<1$, we use Lemma \ref{lemma3.1} $(a)$ with $p=1/s\geq 1$ and we obtain
  \begin{equation}
        \displaystyle\sum_{j=1}^{2d} |a_j|^{\frac{1}{s}}\leq \left(\sum_{j=1}^{2d}|a_j|\right)^{\frac{1}{s}},
    \end{equation} hence
\begin{multline}
	I\leq \int_{\rdd} \exp\Bigg(-\frac{\eps}{(2d)^{1/(2s)}}\left(\sum_{j=1}^d|\sigma_j^{-1}u_j-v_j|^\frac{1}{s}+\sum_{j=d+1}^{2d}|u_j-v_j|^\frac{1}{s}\right)\Bigg) \\ \times 
\exp\Bigg( -\frac{\eps}{(2d)^{1/(2s)}} \left(\sum_{j=1}^d|u_j|^\frac{1}{s}+\sum_{j=d+1}^{2d}|\sigma_{j-d}^{-1}u_j|^{\frac{1}{s}}\right)\Bigg) \dd{u} .
\end{multline}
    We can factorize the previous integral in $2d$ $1$-dimensional integrals as follows. First, we apply Corollary \ref{core0} for the integrals with respect to the variables $u_j,$, $j=1\dots,d$:
    	\begin{multline}
    \int_{\mathbb{R}} \exp\Bigg(-\frac{\eps}{(2d)^{1/(2s)}} |\sigma^{-1} u_j - v_j|^{1/s}\Bigg) \exp\Bigg(-\frac{\eps}{(2d)^{1/(2s)}} |u_j|^{1/s}\Bigg) \, \dd{u_j} \\
    \leq 2\, C\left(\frac{\eps}{(2d)^{1/(2s)}}, s\right) \exp\Bigg(-\frac{\eps}{(2d)^{1/(2s)}} 2^{-(1/s + 1)} |v|^{1/s}\Bigg).
    \end{multline} We stress that
    $$C\left(\frac{\eps}{(2d)^{1/(2s)}}, s\right)=\frac1{(2d)^{1/2}}C(\eps,s).$$
    For $u_{j}$, $j=d+1,\dots,2d$, we use Corollary \ref{core1}:
    \begin{align*}
    \int_{\mathbb{R}} \exp\Bigg(-\displaystyle\frac{\eps}{(2d)^{1/(2s)}} |u_j - v_j|^{1/s}\Bigg)\,& \exp\Bigg(-\displaystyle\frac{\eps}{(2d)^{1/(2s)}} |\sigma^{-1} u_j|^{1/s} \Bigg)\, \dd{u_j} \\
    &\leq C\left(\frac{\eps}{(2d)^{1/(2s)}}, s\right) \exp\Bigg(-\displaystyle\frac{\eps}{(2d)^{1/(2s)}} 2^{-(1/s + 1)} |\sigma^{-1} v_j|^{1/s}\Bigg)\\
    &=  \frac2{(2d)^{1/2}}C(\eps,s) \exp\Bigg(-\displaystyle\frac{\eps}{(2d)^{1/(2s)}} 2^{-(1/s + 1)} |\sigma^{-1} v_j|^{1/s}\Bigg).
    \end{align*}
Therefore, the integral $I$ can be bounded by
   \begin{multline}
   	I\leq \left(\frac{2}{d}\right)^d C(\eps,s)^{2d}  \exp\Bigg(-\displaystyle\frac{\eps}{(2d)^{1/(2s)}} 2^{-(1/s + 1)}\sum_{i=1}^d |v_i|^{1/s} \Bigg) \\ \times\exp\Bigg(-\displaystyle\frac{\eps}{(2d)^{1/(2s)}}2^{-(1/s + 1)}\sum_{i=d+1}^{2d} |\sigma^{-1}v_i|^{1/s}\Bigg).
   	\end{multline}
     Lemma \ref{lemma3.1} $(a)$ now gives the estimate
    \begin{equation}
    \left(\sum_{j=1}^{2d} |a_j|\right)^{\frac{1}{s}} \leq (2d)^{\frac1s -1}\sum_{j=1}^{2d}|a_j|^{\frac{1}{s}},
    \end{equation}
and we obtain 
\begin{align*}
	I&\leq \left(\frac{2}{d}\right)^d  C(\eps,s)^{2d}  \exp\Bigg( \frac{\eps}{(2d)^{1-\frac1{2s}}} 2^{-(1/s + 1)}\left(\sum_{i=1}^d |v_i|+\sum_{i=d+1}^{2d} |\sigma^{-1}v_i|\right)^{1/s}\Bigg)\\
	&= \left(\frac{2}{d}\right)^d  C(\eps,s)^{2d}  \exp\Bigg(-\frac{4\eps}{(2d)^{\frac1{2s}}} \left(\sum_{i=1}^d |v_i|+\sum_{i=d+1}^{2d} |\sigma^{-1}v_i|\right)^{1/s}\Bigg)\\
	&=
	\left(\frac{2}{d}\right)^d  C(\eps,s)^{2d}  \exp\Bigg(-\frac{4\eps}{(2d)^{\frac1{2s}}} |D''v|^{\frac{1}{s}}\Bigg),
\end{align*}
		as desired.
        
   In the case where $s\geq 1$, we use Lemma \ref{lemma3.1} $(b)$ with $p=1/s< 1$ and obtain
    \begin{equation}
    	\sum_{j=1}^{2d} |a_j|^{\frac{1}{s}}\leq 2^{-\frac1s+1}\left(\sum_{j=1}^{2d}|a_j|\right)^{\frac{1}{s}} 
    \end{equation}
(the equality above still holds for $s=1$) so that, arguing as before,
\begin{equation}
	I\leq \left(\frac{2}{d}\right)^d  C(\eps,s)^{2d}   \exp\Bigg(-\frac{\eps}{4(2d)^{1/2}} \sum_{i=1}^d |v_i|^{1/s} \Bigg) \exp\Bigg(- \frac{\eps}{4(2d)^{1/2}}\sum_{i=d+1}^{2d} |\sigma^{-1}v_i|^{1/s}\Bigg).
\end{equation}
  Lemma \ref{lemma3.1} $(b)$  also gives the estimate
\begin{equation}
	\left(\sum_{j=1}^{2d} |a_j|\right)^{\frac{1}{s}} \leq \sum_{j=1}^{2d}|a_j|^{\frac{1}{s}},
\end{equation}
which allows to conclude with the bound
\begin{align*}
	I&\leq \left(\frac{2}{d}\right)^d  C(\eps,s)^{2d}   \exp\Bigg(-\frac{\eps}{4(2d)^{1/2}} \left(\sum_{i=1}^d |v_i|^{1/s}+\sum_{i=d+1}^{2d} |\sigma^{-1}v_i|^{1/s}\right)\Bigg)\\
	&\leq  \left(\frac{2}{d}\right)^d  C(\eps,s)^{2d}   \exp\Bigg(-\frac{\eps}{4(2d)^{1/2}} |D''v|^{\frac{1}{s}}\Bigg),
\end{align*} 
	as claimed.
\end{proof}

The next result is reminiscent of the classical Hanner inequalities, see for instance \cite[Exercise 2.4]{lieb} for a related similar instance. 

\begin{proposition}\label{prop:UC}
Consider $1<p<2$ and set $c_p = p(p-1)$. For all $w,z\in\mathbb R^{d}$ we have
\begin{equation} \label{eq:UC}
|w+z|^{p}+|w-z|^{p}-2|w|^{p} \ge c_{p} |z|^{2}\bigl(|w|+|z|\bigr)^{p-2}. 
\end{equation}
\end{proposition}

\begin{proof}
Fix $w, z \in \mathbb{R}^d$ and define the function $g \colon \mathbb{R} \to \mathbb{R}$ by $g(t) = |w+tz|^p$. The LHS of \eqref{eq:UC} thus coincides with $g(1) + g(-1) - 2g(0)$. We then use Taylor's theorem with integral remainder to expand $g(t)$ around $t=0$, and get
\begin{equation}
g(1) = g(0) + g'(0) + \int_0^1 (1-s) g''(s) \dd{s} , \qquad g(-1) = g(0) - g'(0) + \int_0^{-1} (-1-s) g''(s) \dd{s} .
\end{equation}
Therefore, after suitable substitutions, we obtain 
\begin{equation} g(1) + g(-1) - 2g(0) = \int_0^1 (1-s) \left[ g''(s) + g''(-s) \right] \dd{s} .
\end{equation}
The second derivative of $g(t)$ is
\begin{equation}
g''(t) = p|w+tz|^{p-4} \left( (p-2)((w+tz) \cdot z)^2 + |w+tz|^2 |z|^2 \right). 
\end{equation}
In view of the the Cauchy-Schwarz inequality and the fact that $1<p<2$, we estimate 
\[
(p-2)((w+tz) \cdot z)^2 \ge (p-2)|w+tz|^2 |z|^2.
\]
As a result, we get $g''(t) \ge p(p-1)|w+tz|^{p-2} |z|^2$ for all $t \in \mathbb{R}$. We thus infer
\[
g(1) + g(-1) - 2g(0) \ge p(p-1)|z|^2 \int_0^1 (1-s) \left( |w+sz|^{p-2} + |w-sz|^{p-2} \right) \dd{s} .
\]
For $s \in [0,1]$ we have $|w \pm sz| \le |w| + s|z| \le |w|+|z|$, and since $p-2<0$ we get 
\[ |w+sz|^{p-2} + |w-sz|^{p-2} \ge 2(|w|+|z|)^{p-2}.\] To conclude, we have
\begin{align}
g(1) + g(-1) - 2g(0) & \ge 2p(p-1)|z|^2 (|w|+|z|)^{p-2}\int_0^1 (1-s)  \dd{s} \\ 
&= p(p-1)|z|^2 (|w|+|z|)^{p-2},
\end{align} that is the claim. 
\end{proof}

\begin{remark}
Considering the endpoint cases, the same integral argument with $p=2$ yields equality $|w+z|^{2}+|w-z|^{2}-2|w|^{2}=2|z|^{2}$, while for $p=1$ the RHS is identically zero.
\end{remark}

\begin{proposition}\label{prop-convbound}
Consider $\eps>\delta>0$ and $p>0$. There exist constants $C=C_{d,\eps,p}>0$ and $K=K_{d,\eps,\delta,p}>0$ such that, for all $x \in \rd$, 
\[
I_{\eps,p}(x) \coloneqq \int_{\R^{d}} e^{-\eps\bigl(|t|^{p}+|t-x|^{p}\bigr)} dt \le \begin{cases} C(1+|x|)^{\frac{d(2-p)}{2}}\, e^{-\eps\,2^{1-p}|x|^{p}}, & p>1\\ 
C(1+|x|)^d e^{-\eps|x|}, & p=1  \\
Ke^{-\delta|x|^p}, & 0<p<1.
\end{cases}
\]
\end{proposition}

\begin{proof} We separately discuss the proofs depending on the value of $p$, which actually entails exploitation of convexity or concavity of the map $s\mapsto s^p$, $s\ge 0$. 

\noindent \textbf{Case $p>1$.} Define $\Psi_{p,x}(t) \coloneqq |t|^p+|t-x|^p-2^{1-p}|x|^p$, $t \in \rd$. A straightforward application of Jensen's inequality shows that $\Psi_{p,x}(t)\ge 0$. We can clearly recast the integral as
\[ 
I_{\eps,p}(x) = e^{-\eps2^{1-p}|x|^p} \int_{\R^d} e^{-\eps \Psi_{p,x}(t)} \dd t.
\]
The problem is now reduced to finding an upper bound for the remaining integral, which we denote by $J(x)$. To this aim, set $u = t - x/2$, so that 
\[
J(x) = \int_{\R^d} e^{-a \Psi_{p,x}(t)} \dd{t}  = \int_{\R^d} e^{-a \left( |u+x/2|^p + |u-x/2|^p - 2|x/2|^p \right)}  \dd{u} .
\]
With a slight abuse of notation, we denote the exponent term in the new variable as $\Psi(u) \coloneqq |u+x/2|^p + |u-x/2|^p - 2|x/2|^p$. We can now resort to Proposition \ref{prop:UC} with $w = x/2$ and $z = u$, hence obtaining
\[
J(x) \le \tilde{J}(x)\coloneqq \int_{\R^d} \exp\Big(-\eps p(p-1)|u|^2 \Big( |x|/2 + |u| \Big)^{p-2} \Big)  \dd{u} . 
\]
Let us now split the domain $\R^d$ into two regions $B_1 = \{u \in \R^d : |u| \le |x|/2 \}$ and $B_2 = \{u \in \R^d : |u| > |x|/2 \}$. The integral is then split into $\tilde{J}(x) = \tilde{J}_1(x) + \tilde{J}_2(x)$ accordingly. 

Let us focus on $\tilde{J}_1(x)$ first. For $u \in B_1$ we have $|x|/2 \le {|x|}/{2} + |u| \le |x|$. Since $s \mapsto s^{p-2}$ is non-increasing for $s \ge 0$, we have $({|x|}/{2} + |u|)^{p-2} \ge |x|^{p-2}$. We thus conclude that
\begin{align*}
\tilde{J}_1(x) & \le \int_{B_1} \exp\left(-\eps \,p(p-1)|x|^{p-2}|u|^2\right) \dd{u}  \\
&\le \int_{\R^d} \exp\left(-\eps\,p(p-1)|x|^{p-2}|u|^2\right) \dd{u}  \\
& \le C_1 |x|^{\frac{d(2-p)}{2}},
\end{align*} for a suitable $C_1>0$ that does not depend on $x$. This argument holds for $x\neq 0 $ --- in fact, for $x=0$ we have $J_1(0)=0$.

On the other hand, if $u \in B_2$ we have $|u| > |x|/2$, which in turn implies $|x|/{2} + |u| < 2|u|$. This leads to $({|x|}/{2} + |u|)^{p-2} > (2|u|)^{p-2}$. We thus infer
\[
\tilde{J}_2(x) = \int_{B_2} e^{-\eps \Psi(u)}\, \dd{u}  \le \int_{\R^d} e^{-\eps (p-1)2^{p-2} |u|^p}  \dd{u}  \eqqcolon C_2, 
\] for a suitable finite constant $C_2>0$.

We have therefore established a global bound for the integral $J(x)$ that is valid for all $x \in \R^d$, namely $J(x) \le C (1+|x|)^{\frac{d(2-p)}{2}}$, for a suitable $C>0$ independent of $x$. Then
\[
I_{\eps,p}(x) = e^{-\eps 2^{1-p}|x|^p} J(x) \le C(1+|x|)^{\frac{d(2-p)}{2}} e^{-a\,2^{1-p}|x|^{p}},
\] that is the claim. 

\bigskip

\noindent \textbf{Case $p=1$.} We begin by making the change of variables $u = t - x/2$, so the integral becomes
\begin{equation}
    I_{\eps,1}(x) = \int_{\R^d} e^{-\eps\bigl(|u+x/2| + |u-x/2|\bigr)} \,  \dd{u} .
\end{equation}
The term in the exponent can be bounded from below using the triangle inequality:
\[ 
|u+x/2| + |u-x/2| \ge |(u+x/2) \pm (u-x/2)| \ge \max\{ |x|, 2|u| \}, 
\]
so we are now left with bounding 
\begin{equation}
    I_{\eps,1}(x) \le \int_{\R^d} e^{-\eps \max\lbrace |x|, 2|u| \rbrace}  \dd{u} .
\end{equation}
We split this integral over $\R^d$ into two regions: the ball $B=B(0, |x|/2)$, centered at the origin with radius $|x|/2$, and its complement $B^c=\R^d \setminus B(0, |x|/2)$.

For $u \in B$ we have $|u| \le |x|/2$, which implies $2|u| \le |x|$. The integral over this region thus becomes
\begin{equation}
    e^{-\eps|x|}\int_{B(0, |x|/2)}   \dd{u}  = C_d |x|^d e^{-\eps|x|},
\end{equation} for a suitable $C_d>0$. 

If $u \in B^c$ instead, we have $2|u| > |x|$, and $ \int_{|u| > |x|/2} e^{-2\eps|u|}  \dd{u}  \le \int_{\R^d} e^{-2\eps|u|}  \dd{u}  < C_d'$ clearly holds for a suitable $C_d'>0$. 

\bigskip 
\noindent \textbf{Case $0<p<1$.} We actually prove a slightly stronger result, that is: for every $0<\lambda<1$ and $x \in \rd$ there exists $C(\lambda)>0$ independent of $x$ such that $ I_{\eps,p}(x) \le C e^{-\lambda \eps |x|^p}$, from which the claim immediately follows. 

We first split the exponent in the integrand by leveraging on the parameter $\lambda \in (0,1)$:
\[ I_{\eps,p}(x) = \int_{\R^d} \exp\left(-a\lambda\left(|t|^p + |t-x|^p\right)\right)  \exp\left(-a(1-\lambda)\left(|t|^p + |t-x|^p\right)\right) \dd{t} . \]

Concerning the first factor, we preliminarily note that a combination of concavity, monotonicity and triangle inequality shows that 
\[ |t|^p+|x-t|^p \ge (|t|+|x-t|)^p  \ge |x|^p.
\] Therefore, 
\[ \exp\left(-\eps\lambda\left(|t|^p + |t-x|^p\right)\right) \le \exp\left(-\eps\lambda |x|^p\right). \]
Bounding the second factor trivially amounts to use that $|t-x|^p \ge 0$:
    \[
        \exp\left(-\eps(1-\lambda)\left(|t|^p + |t-x|^p\right)\right) \le \exp\left(-\eps(1-\lambda)|t|^p\right).
    \]
To sum up, we have obtained
\begin{align*}
    I_{\eps,p}(x) &\le \int_{\R^d} \exp\left(-\eps\lambda|x|^p\right) \cdot \exp\left(-\eps(1-\lambda)|t|^p\right) \,\dd{t}  \\
    &= \exp\left(-\eps\lambda|x|^p\right) \int_{\R^d} \exp\left(-\eps(1-\lambda)|t|^p\right) \,\dd{t} .
\end{align*}
The remaining integral is a finite constant for any $\lambda \in (0,1)$, and the claim follows
\end{proof}

\subsection{Optimally combining bounds}

Consider a function $F \colon \mathbb{R}^{2d} \to \mathbb{C}$ that satisfies the following two inequalities for some given constants $C, k > 0$, $m\ge 0$, and $p>0$:
\begin{align}
|F(x,y)| &\leq C (1+|x|)^m e^{-k|x|^p},\quad x,y\in\rd, \label{eq:bound_x} \\
|F(x,y)| &\leq C (1+|y|)^m e^{-k|y|^p}, \quad x,y\in\rd \label{eq:bound_y}.
\end{align} For future purposes, our goal is to find a bound of the form
$$|F(z)| \leq C(z) e^{-b|z|^p}, \quad z=(x,y) \in \mathbb{R}^{2d},$$
where  $C(z)$ is a positive function with at most polynomial growth in $|z|$, and $b > 0$ is the largest possible constant. 

\begin{proposition}
\label{prop-optbound}
Given the bounds in \eqref{eq:bound_x} and \eqref{eq:bound_y}, the function $F(z) = F(x,y)$ satisfies the inequality:
$$|F(z)| \lesssim(1+|z|)^m e^{-b|z|^p},$$
where the largest possible value for the constant $b>0$ is
$$b = 2^{-p/2}k$$
\end{proposition}

\begin{proof}
First, we combine \eqref{eq:bound_x} and \eqref{eq:bound_y} into a single bound, that is: 
$$|F(x,y)| \lesssim\left(1+|z|\right)^m e^{-k \max(|x|^p, |y|^p)}.$$
In fact, the hypoteses on $F$ imply
    \begin{align}
\frac{|F(x,y)|}{(1+|x|+|y|)^m} &\lesssim \bigg[\frac{(1+|x|)}{(1+|x|+|y|)}\bigg]^m e^{-k|x|^p}\lesssim e^{-k|x|^p}\\
\frac{|F(x,y)|}{(1+|x|+|y|)^m} &\lesssim \bigg[\frac{(1+|y|)}{(1+|x|+|y|)}\bigg]^m e^{-k|y|^p} \lesssim e^{-k|y|^p}.
\end{align}
So, \begin{equation*}
    \frac{|F(x,y)|}{(1+|x|+|y|)^m} \leq C e^{-k\max(|x|^p,|y|^p)}.
\end{equation*}
In particular \begin{equation*}
    |F(x,y)|\lesssim (1+|x|+|y|)^me^{-k\max{(|x|^p,|y|^p)}}\lesssim (1+|z|)^me^{-k\max(|x|^p,|y|^p)}.
\end{equation*}

Now, for the desired bound to hold for all $x,y$, we must have $k \max(|x|^p, |y|^p) \geq b|z|^p$, hence $b$ must satisfy
$$ b \leq k \frac{\max(|x|^p, |y|^p)}{|z|^p} \quad \forall x,y \in \mathbb{R}^d \text{ not both zero.} $$
Let then $u,v$ be the norms $|x|$ and $|y|$ respectively. We claim that the minimum value of the ratio
$$ g(u,v) = \frac{(\max(u,v))^p}{(u^2+v^2)^{p/2}} $$
is $2^{-p/2}$. In fact, the function $g(u,v)$ is homogeneous of degree 0, so its value depends only on the ratio $v/u$. Passing to polar coordinates, namely setting $u = r\cos\theta$ and $v=r\sin\theta$ for $r>0$ and $\theta \in [0, \pi/2]$, we have
\begin{align*}
g(u,v) &= \frac{(\max(r\cos\theta, r\sin\theta))^p}{((r\cos\theta)^2 + (r\sin\theta)^2)^{p/2}} 
= \frac{r^p (\max(\cos\theta, \sin\theta))^p}{(r^2(\cos^2\theta+\sin^2\theta))^{p/2}} \\
&= (\max(\cos\theta, \sin\theta))^p.
\end{align*}
The problems boils down to finding the minimum of $h(\theta) = (\max(\cos\theta, \sin\theta))^p$ for $\theta \in [0, \pi/2]$. Elementary arguments show that the minimum occurs at $\theta=\pi/4$ (so that $u=v$, i.e., $|x|=|y|$), and the minimum value is 
$$ \min_{\theta \in [0, \pi/2]} h(\theta) = \left(\frac{1}{\sqrt{2}}\right)^p = 2^{-p/2}, $$ as claimed. 
\end{proof}

\section{Proof of the main results}\label{sec-main}

\subsection{Proof of Theorem \ref{stftdecay} }
Set $p=1/s$, so that if $s\ge 1/2$ we have $p \le 2$. Recall the notation introduced in Proposition \ref{prop-convbound}:
\[
I_{\eps,p}(x)=\int_{\R^{d}} e^{-\eps\big(|t|^{p}+|t-x|^{p}\big)} \dd{t} . 
\] 
Let us focus on the convex case $1 < p \le 2$ (equivalently, $1/2 \le s < 1)$ for the sake of clarity, since the proof in the remaining cases follows by the same arguments with slight modifications --- in fact, resorting to the different regimes in Proposition \ref{prop-convbound}.

Set $k = 2^{1-p}\eps $ and $\gamma = \frac{d(2-p)}{2}$. First, we note that the assumptions and Proposition \ref{prop-convbound} imply
\begin{equation*}
     |V_gf(x,\xi)| \le \int_{\R^d} |f(y)||g(y-x)| dy  \lesssim I_{\eps,p}(x)  \lesssim_{d,\eps,p} (1+|x|)^{\gamma} e^{-k|x|^p}.
\end{equation*}
On the other hand, in light of Proposition \ref{STFTprop} we have
\begin{equation*} |V_g f(x,\xi)| = |V_{\hat{g}}\hat{f}(\xi,-x)| \le \int_{\R^d} |\hat{f}(\xi)||\hat{g}(\xi-y)| dy \lesssim I_{\eps,p}(\xi) \lesssim  (1+|\xi|)^{\gamma} e^{-k|\xi|^p}.
\end{equation*}
We can now resort to Proposition \ref{prop-optbound} and conclude that
\[ |V_gf(z)| \lesssim (1+|z|)^{\gamma} e^{-2^{1-3p/2} \eps |z|^p}, \qquad z \in \rdd. \] 

To characterize the spectral decay of $V_gf$ we use Proposition \ref{STFTprop} again: setting $\zeta =(\omega,\eta) \in \rdd$, we have
\begin{equation} |\widehat{V_gf}(\omega,\eta)| =|f(-\eta)| |\overline{\hat g(\omega)}| \lesssim  e^{-\eps(|\omega|^{1/s} + |\eta|^{1/s})} \le e^{-\eps |\zeta|^{1/s}}. 
\end{equation} 

To sum up,  for every $z=(x,\xi) \in \rdd$ and $\zeta=(\omega,\eta) \in \rdd$, we have
\[ 
|V_gf(z)| \lesssim (1+|z|)^{\gamma} e^{-2^{1-3/(2s)} \eps |z|^p}, \qquad |\widehat{V_gf}(\zeta)| \lesssim e^{-\eps |\zeta|^{1/s}}. 
\]
To take into account the polynomial prefactor, except for the case where $s=1/2$ (hence $\gamma=0$), we must absorb it into the exponential and allow for arbitrarily small losses: for any $\alpha > 0$ we can write $(1+|z|)^\gamma \lesssim_\alpha e^{\alpha |z|^{1/s}}$. As a result, for every $0<\delta<2^{1-3/(2s)} \eps$ we have
\[ 
|V_gf(z)| \lesssim e^{-\delta |z|^{1/s}}, \qquad |\widehat{V_gf}(\zeta)| \lesssim e^{-\eps |\zeta|^{1/s}}. 
\]
Since $2^{1-3/(2s)}<1$ in the range $1/2 \le s < 1$, we conclude that $V_gf \in S^s_{s,\delta}(\rdd)$. 

We conclude this section by transferring  the STFT regularity stated in Theorem \ref{stftdecay} to the Wigner distribution.

\begin{corollary}\label{Wignerdecay}
Let $f, g \in \cS^s_{s,\eps}(\rd)$ with $s \ge \frac{1}{2}$ and $\eps>0$. Then $W(f,g)$ belongs to the space $\cS^s_{s,\delta }(\rdd)$ for every $\delta$ satisfying \eqref{delta-stft}.
\end{corollary}

\begin{proof}
     Using \eqref{Wig-stft}  and Theorem \ref{stftdecay}    we can write
     \begin{equation}\label{e-Wigner}|W(f,g)(z)|=2^d|V_{\cI g}f(2z)|\lesssim e^{-\delta 2^{1/s}|z|^{1/s}}, \quad z\in\rdd,\end{equation}
     whereas the \ft \, of the cross-Wigner in \eqref{trasf-Wigner} is (cf.\ \cite[Lemma 1.3.11]{Elenabook})      $$|\cF W(f,g)(y,\eta)|=|V_gf(-\eta,y)|$$
     with Theorem \ref{stftdecay} yields
      $$|\cF W(f,g)(z)|\lesssim e^{-\delta |z|^{1/s}},\quad z\in\rdd.$$  
      Since $2^{1/s}\delta >\delta$ , the thesis follows.
      
\end{proof}
\subsection{Proof of Theorem \ref{viceversaSTFT} }
Fix the window $g\in\cS^s_{s,\eps}(\rd)\setminus \{0\}$. Using the connection between  the Wigner distribution and the  STFT in \eqref{Wig-stft} we infer
\begin{equation}
    |W(f,\cI g)(z)|\lesssim e^{-2\eps|z|^{1/s}},\quad \forall z\in\rdd.
\end{equation}
In particular, we obtain that  $W(f,\cI g)\in\cS(\rd)$. By assumption   $g\in\cS^s_{s,\eps}(\rd)$, which implies  $\cI g\in\cS^s_{s,\eps}(\rd)\subset\cS(\rd)$, therefore $f\in\cS(\rd)$.  We can apply Proposition \ref{Marginals} and obtain:\begin{equation*}
    |f(x)||\cI g(x)|=|f(x)\overline{\cI g(x)}|\leq \int_{\rdd} |W(f,\cI g)(x,\xi)|\dd\xi    \lesssim \int_{\rdd} \!\!e^{-2\eps(|x|^{1/s}+|\xi|^{1/s})}\dd\xi    \lesssim  e^{-2\eps |x|^{1/s}}
\end{equation*}
for every $x\in\rd$.
Hence, 
\begin{equation}   \norm{f(\cdot)e^{\eps|\cdot|^{1/s}}}_{\i}\norm{\overline{\cI g(\cdot)}e^{\eps|\cdot|^{1/s}}}_{\i}<\i.
\end{equation}
Hence, by using that $\norm{\overline{\cI g(\cdot)}e^{\eps|\cdot|^{1/s}}}_{\i}$ is finite and non-zero, we have that:\begin{equation}
    \norm{f(\cdot)e^{\eps|\cdot|^{1/s}}}_{\i}<\i.
\end{equation}
The other marginal property in Proposition \ref{Marginals} gives 
\begin{equation}
    \hat f(\xi)\overline{\widehat{\cI g}(\xi)}=\int_{\rd} W(f,\cI g)(x,\xi)\dd x,\quad \forall \xi\in\rd,
\end{equation}
and the same argument as above yields the estimate $
     \norm{\hat f(\cdot)e^{\eps|\cdot|^{1/s}}}_{\i}<\i$.

\subsection{Proof of Theorem \ref{teoe00}}
 We argue as in the proof of \cite[Theorem 1.1]{spreading}. Namely, the covariance property of the Wigner distribution in \eqref{eq-wigcov} and its symplectic covariance in  \eqref{eq-sympcovwig} justify the following computations:
    \begin{align*}
        |\la \wh S \pi(z)g,\pi(w)\gamma\ra|^2=&\la W(\wh S \pi(z)g),W(\pi(w)\gamma)\ra\\
       =&\int_{\rdd}W(\wh S \pi(z)g)(t)\overline{W(\pi(w)\gamma)(t)}\dd{t} .\\
       =&\int_{\rdd}Wg(S^{-1}t-z)\overline{W\gamma(t-w)}\dd{t} \\
       =&\int_{\rdd} Wg(S^{-1}u+S^{-1}w-z)\overline{W\gamma(u)} \dd{u} .
    \end{align*} 
   If $f\in\cS^s_{s,\eps}(\rd)$  then   $|Wf(z)|\lesssim e^{-2^{\frac1s}\delta|z|^{\frac1s}}$ by \eqref{e-Wigner}, where $\delta$ is defined  in \eqref{delta-stft}.
Therefore, we control the integrand above as \begin{equation}
     |Wg(S^{-1}u+S^{-1}w-z)|\lesssim e^{-2^{\frac1s}\delta|S^{-1}u+S^{-1}w-z|^{\frac1s}},\qquad |W\gamma(u)|\lesssim e^{-2^{\frac1s}\delta|u|^{\frac1s}}.
\end{equation}
This gives the estimate \begin{equation}
    |\la \wh S \pi(z)g,\pi(w)\gamma\ra|^2\lesssim \int_{\rdd}e^{-2^{\frac1s}\delta|S^{-1}u+S^{-1}w-z|^{\frac1s}}e^{-2^{\frac1s}\delta|u|^{\frac1s}} \dd{u} .
\end{equation}

By using the decomposition $S=U^\top DV$, $S^{-1}=V^\top D^{-1}U$ we write \begin{align*}
    \int_{\rdd}e^{-2^{\frac1s}\delta|S^{-1}u+S^{-1}w-z|^{\frac1s}}e^{-2^{\frac1s}\delta|u|^{1/s}} \dd{u} =&\int_{\rdd}e^{-2^{\frac1s}\delta|D^{-1}u+V(S^{-1}w-z)|^{\frac1s}}e^{-2^{\frac1s}\delta|u|^{\frac1s}} \dd{u} \\
    =&\det(\Sigma)^{-1}\int_{\rdd}e^{-2^{\frac1s}\delta|D'u'-v|^{\frac1s}}e^{-2^{\frac1s}\delta|D''u'|^{\frac1s}} \dd{u'}
\end{align*}
where we used the change of variable $u=D''u'$ and we defined  $v\coloneqq -V(S^{-1}w-z)$. Finally, applying Lemma \ref{lemmatec2d} we obtain the result.

\begin{proof}[\bf Proof of Corollary \ref{cor0e}]

	Observe that \begin{align*}
	    \sigma_{\min}|\omega|&=\sigma_1^{-1}|\omega|=\sigma_1^{-1}|U(\omega)|=\sigma_1^{-1}|D'^{-1}D'U(\omega)|\\
        &\leq \sigma_1^{-1}\norm{D'^{-1}}\,|D'U(\omega)|=|D'U(\omega)|.
	\end{align*} Then,  the thesis follows from Theorem \ref{teoe00} with $z=0$ .
\end{proof}

\section{Gelfand-Shilov confinement for metaplectic operators}\label{sec-metdecay}
In this section we aim to understand the Gelfand-Shilov behaviour of $f\in\cS^s_{s,\eps}(\rd)$ under the action of a metaplectic operator $\wh S$. We focus first on the two two regimes $1/2< s<1$ and $s\geq 1$ --- for the sake of clarity we treat them separately, even if the demonstration techniques are the same. The Gaussian case $s=1/2$ will be treated subsequently.

\subsection{The case $1/2<s<1$}

In the following statements we write $\sigma_{\min}=\sigma_1^{-1}\le 1$ to denote the smallest singular value of $S$ --- cf.\ Section \ref{subsec:metaplectic}.

\begin{proposition} \label{metadecay1}
	Assume $1/2< s<1$ and $\eps>0$. If $\wh S \in \Mp(d,\bR)$ and $f\in\cS^s_{s,\eps}(\rd)$ we have:
    \begin{itemize}
	    \item[(i)] If $\sigma_{\min}\leq 2^{-s}\sqrt{d}$, then $\wh S f\in\cS^s_{s,\delta}(\rd)$ for every  $\delta<d^{-1/(2s)} 2^{1-1/s} \sigma_{\min}^{1/s}\eps$.
        \item[(ii)] If $\sigma_{\min}> 2^{-s}\sqrt{d}$, then $\wh S f\in\cS^s_{s,\eps}(\rd)$.
	\end{itemize}
\end{proposition}

\begin{proof}
We fix a window $g\in\cS^s_{s,\eps}(\rd)\setminus\{0\}.$ We can rephrase the estimate in Corollary \ref{cor0e} as \begin{equation*}
    |V_g\wh S f(z)|\lesssim e^{-(\frac{2}{d})^{\frac{1}{2s}}2\sigma_{\min}^{\frac1s}\delta|z|^{\frac1s}},\quad \forall z\in\rdd,\ \forall \delta<2^{1-\frac{3}{2s}}\eps.
\end{equation*} After dividing and multiplying by $2^{1-\frac1s}$ in the exponent we get 
\begin{equation*}
    |V_g\wh S f(z)|\lesssim e^{-(2^{1-\frac1s})(2^{\frac1s-1})(\frac{2}{d})^{\frac{1}{2s}}2\sigma_{\min}^{\frac1s}\delta|z|^{\frac1s}}=e^{-(2^{1-\frac1s})(2\sqrt{\frac{2}{d}}\sigma_{\min})^{\frac1s}\delta|z|^{\frac1s}},
\end{equation*}
for every $ z\in\rdd$, for every $  \delta<2^{1-\frac{3}{2s}}\eps.$ Let us  compare the coefficients $(2\sqrt{\frac{2}{d}}\sigma_{\min})^{\frac1s}\delta$ and $\eps$. We distinguish between  two cases.\\

\emph{First case:} \begin{equation*}
        (2\sqrt{\frac{2}{d}}\sigma_{\min})^{\frac1s}\delta\leq \eps,\quad  \delta<2^{1-\frac{3}{2s}}\eps.
    \end{equation*}
    This happens if and only if \begin{equation*}\displaystyle\sup_{\delta<2^{1-\frac{3}{2s}}\eps}(2\sqrt{\frac{2}{d}}\sigma_{\min})^{\frac1s}\delta =2(\frac{1}{\sqrt{d}}\sigma_{\min})^{\frac1s}\eps\leq \eps,\end{equation*} which entails the condition $\sigma_{\min}\leq 2^{-s}\sqrt{d}$. In this case we have $\cS^s_{s,\eps}(\rd)\subset \cS^s_{s,(2\sqrt{\frac{2}{d}}\sigma_{\min})^{\frac1s}\delta}(\rd)$ for every $\delta<2^{1-\frac{3}{2s}}\eps$. In particular, the window $g$ is in the larger Gelfand-Shilov class and we apply Theorem \ref{viceversaSTFT} to deduce that \[ \wh S f\in\cS^s_{s,(2\sqrt{\frac{2}{d}}\sigma_{\min})^{\frac1s}\delta}(\rd),\quad  \forall \delta<2^{1-\frac{3}{2s}}\eps.\]

\emph{Second case:} $$\sigma_{\min}>2^{-s}\sqrt{d}.$$ This implies that there exists $ \delta<2^{1-\frac{3}{2s}}\eps$ such that $(2\sqrt{\frac{2}{d}}\sigma_{\min})^{\frac1s}\delta>\eps$, therefore we have 
    \begin{equation*}
    |V_g\wh S f(z)|\lesssim e^{-(2^{1-\frac1s})\eps|z|^{\frac1s}},\quad \forall z\in\rdd.
\end{equation*}
By Theorem \ref{viceversaSTFT},  we infer $\wh S f\in\cS^s_{s,\eps}(\rd)$.
\end{proof}

\begin{remark}
 If $2^{-s}\sqrt{d}>1$, since the smallest singular value of the symplectic projection $S$ of $\wh S $ satisfies $\sigma_{\min}\leq 1$,  we immediately fall in the first case. This happens, for instance, if $d\geq 4$. In fact, $1/2\leq s<1$, so $2^{-s}>1/2$, hence \begin{equation*}
        2^{-s}\sqrt{d}>\frac{\sqrt{d}}{2}\geq \frac{\sqrt{4}}{2}=1.
    \end{equation*}
\end{remark}
    
\begin{corollary}\label{GSVgSf1}
    Assume $1/2< s<1$, $\eps>0$, $f,g\in\cS^s_{s,\eps}(\rd)$ and $\wh S \in \Mp(d,\bR)$.  Then the STFT $V_gf$ enjoys the following properties: 
    \begin{itemize}
        \item [(i)] If $\sigma_{\min}\leq 2^{-2s+1}\sqrt{d}$, then $V_g(\wh S f)\in \cS^s_{s,\delta}(\rdd)$ for all $\delta<d^{-1/2s} 2^{1-1/s} \sigma_{\min}^{1/s}\eps$.
        \item[(ii)] If $\sigma_{\min}>2^{-2s+1}\sqrt{d}$, then $V_g(\wh S f)\in \cS^s_{s,\eps}(\rdd)$.
    \end{itemize} 
\end{corollary}
\begin{proof}
    Since $f,g\in\cS^s_{s,\eps}(\rd)$, by Corollary \ref{cor0e} we have the estimate
    \begin{equation*}
        |V_g\wh S f(z)|\lesssim e^{-2(\sqrt{\frac{2}{d}}\sigma_{\min})^{\frac1s}\delta|z|^{\frac1s}},\quad \forall z\in\rdd,\ \forall \delta< 2^{1-\frac{3}{2s}}\eps.
    \end{equation*}
    In view of \eqref{FTSTFT} we get 
    \begin{equation*}
        |\widehat{V_g\wh S f}(\omega,\eta)|=|\wh S f(-\eta)||\hat g(\omega)|\lesssim |\wh S f(-\eta)|e^{-\eps|\omega|^{\frac1s}},\quad \forall (\omega,\eta)\in\rdd.
    \end{equation*}
      \noindent  \emph{Case 1:} $\sigma_{\min}\leq 2^{-s}\sqrt{d}$. \\
        By Proposition \ref{metadecay1}, \begin{equation*}
        |\wh S f(-\eta)|\lesssim e^{-(2\sqrt{\frac{2}{d}}\sigma_{\min})^{\frac1s}\delta|\eta|^{\frac1s}},\quad \forall \eta\in\rd,\ \forall \delta< 2^{1-\frac{3}{2s}}\eps.
    \end{equation*}
    However, the condition $\sigma_{\min}\leq 2^{-s}\sqrt{d}$ implies that $(2\sqrt{\frac{2}{d}}\sigma_{\min})^{\frac1s}\delta\leq \eps$ for every $ \delta<2^{1-\frac{3}{2s}}\eps$. Therefore,  \begin{equation*}
        |\hat g(\omega)|\lesssim e^{-(2\sqrt{\frac{2}{d}}\sigma_{\min})^{\frac1s}\delta|\omega|^{\frac1s}},\quad \forall y\in\rd,\ \forall \delta< 2^{1-\frac{3}{2s}}\eps,
    \end{equation*}
    and thus 
    \begin{equation*}
        |\widehat{V_g\wh S f}(\zeta)|\lesssim  e^{-(2\sqrt{\frac{2}{d}}\sigma_{\min})^{\frac1s}\delta|\zeta|^{\frac1s}},\quad \forall \zeta\in\rdd,\ \forall \delta< 2^{1-\frac{3}{2s}}\eps.
    \end{equation*}
    Now, since $1/2<s<1$, we have that $2^{\frac1s}>2$ and $V_g(\wh S f)$ decays faster than its Fourier transform. We conclude that $V_g (\wh S f) \in\cS^s_{s,2(\sqrt{\frac{2}{d}}\sigma_{\min})^{\frac1s}}(\rdd)$.\\
    \noindent
\emph{Case 2:} $\sigma_{\min}>2^{-s}\sqrt{d}$. \\
    Again, by Proposition \ref{metadecay1} and \eqref{FTSTFT} we get 
    \begin{equation*}
         |\widehat{V_g\wh S f}(\zeta)|\lesssim e^{-\eps|\zeta|^{\frac1s}},\quad \forall \zeta\in\rdd.
    \end{equation*}
    We must compare the coefficients $\eps$ and $2(\sqrt{\frac{2}{d}}\sigma_{min})^{\frac1s}\delta$ for $\delta<\eps 2^{1-\frac{3}{2s}}$. Let us separately discuss the following two sub-cases. \par    \emph{Case 2a:} $2^{-s}\sqrt{d}<\sigma_{\min}\leq 2^{-2s+1}\sqrt{d}$. \\
 Note that condition $1/2\leq s<1$ implies $2^{-s}\sqrt{d}<2^{-2s+1}\sqrt{d}$, therefore, this is not an empty scenario. Furthermore,  the following condition holds:
\begin{equation*}
\displaystyle\sup_{\delta<2^{1-\frac{3}{2s}}\eps}2(\sqrt{\frac{2}{d}}\sigma_{\min})^{\frac1s}\delta\leq \eps.
    \end{equation*}
    We conclude as before that $V_g(\wh S f)\in\cS^s_{s,2(\sqrt{\frac{2}{d}}\sigma_{\min})^{\frac1s}}(\rdd)$.\par
    \emph{Case 2b:}  $\sigma_{\min}>2^{-2s+1}\sqrt{d}$. \\ Under this assumption, there exists $\delta<2^{1-\frac{3}{2s}}\eps$ such that $2(\sqrt{\frac{2}{d}}\sigma_{\min})^{\frac1s}\delta\geq \eps$. So, $V_g(\wh S f)\in\cS^s_{s,\eps}(\rdd)$.
    This concludes the proof.
\end{proof}

\begin{remark}
   If $d\geq 4$, the condition $\sigma_{\min}>2^{-2s+1}\sqrt{d}$ never holds. In fact, for $1/2< s<1$ we have $2^{-2s+1}\sqrt{d}>\frac{\sqrt{d}}{2}$. 
\end{remark}
\subsection{The case $s\geq 1$}

\begin{proposition}\label{metadecay2}
    Assume $s\geq1$, $\eps>0$. If $\wh S \in \Mp(d,\bR)$ and $f\in\cS^s_{s,\eps}(\rd)$ we have $\wh S f\in\cS^s_{s,\frac{1}{16\sqrt{2d}}(4\sigma_{\min})^{\frac1s}\delta}(\rd)$ for every $ \delta<2^{-\frac{1}{2s}}\eps$.
\end{proposition}

\begin{proof}
    Fix a window $g\in \cS^s_{s,\eps}(\rd)\setminus\{0\}$. By virtue of Corollary \ref{cor0e} we infer \begin{equation*}
        |V_g \wh S f(z)|\lesssim e^{-\frac{1}{8\sqrt{2d}}(2\sigma_{\min})^{\frac1s}\delta|z|^{\frac1s}},\quad \forall z\in\rdd,\ \forall \delta<2^{-\frac{1}{2s}}\eps. 
    \end{equation*}
    After dividing and multiplying by $2^{\frac1s-1}$ in the exponent we get
    \begin{equation*}
         |V_g\wh S f(z)| \lesssim e^{-(2^{1-\frac1s})(2^{\frac1s-1})\frac{1}{8\sqrt{2d}}(2\sigma_{\min})^{\frac1s}\delta|z|^{\frac1s}} = e^{-(2^{1-\frac1s})\frac{1}{16\sqrt{2d}}(4\sigma_{\min})^{\frac1s}\delta|z|^{\frac1s}},
    \end{equation*}
    for every $ z\in\rdd$, and $ \delta<2^{-\frac{1}{2s}}\eps.$
    Observe that
    \begin{align*}
\displaystyle\sup_{\delta<2^{-\frac{1}{2s}}\eps}\frac{1}{16\sqrt{2d}}(4\sigma_{\min})^{\frac1s}\delta&=\frac{1}{16\sqrt{2d}}(4\sigma_{\min})^{\frac1s}2^{-\frac{1}{2s}}\eps= \frac{1}{16\sqrt{2d}}(\frac{4}{\sqrt{2}}\sigma_{\min})^{\frac1s}\eps\\ &\leq \frac{1}{16\sqrt{2d}}(\frac{4}{\sqrt{2}})^{\frac1s}\eps,
    \end{align*}
    and the assumption $s\geq 1$ yields
    \begin{equation*}
        \frac{1}{16\sqrt{2d}}(\frac{4}{\sqrt{2}})^{\frac1s}\eps\leq \frac{4}{16\sqrt{4d}}\eps\leq \eps.
    \end{equation*}
    Therefore $g\in \cS^s_{s,\frac{1}{16\sqrt{2d}}(4\sigma_{\min})^{\frac1s}\delta}(\rd)$ for every $ \delta<2^{-\frac{1}{2s}}\eps$ and the claim follows by Theorem \ref{viceversaSTFT}.
\end{proof}

\begin{corollary}\label{GSVgSf2}
    Given $s\geq 1$, $\eps>0$, $f,g\in\cS^s_{s,\eps}(\rd)$ and $\wh S \in \Mp(d,\bR)$, we have $V_g\wh S f\in \cS^s_{s,\frac{1}{16\sqrt{2d}}(4\sigma_{\min})^{\frac1s}\delta}(\rdd)$ for every $\delta<2^{-\frac{1}{2s}}\eps$.
\end{corollary}

\begin{proof}
    Since $f,g\in\cS^s_{s,\eps}(\rd)$, by Corollary \ref{cor0e} we infer
    \begin{equation*}
        |V_g\wh S f(z)|\lesssim e^{-\frac{1}{8\sqrt{2d}}(2\sigma_{\min})^{\frac1s}\delta|z|^{\frac1s}},\quad \forall z\in\rdd,\ \forall \delta<2^{-\frac{1}{2s}}\eps.
    \end{equation*}
    Now, combining \eqref{FTSTFT} with Proposition \ref{metadecay2} we obtain
    \begin{equation*}
        |\widehat{V_g\wh S f}(y,\eta)|=|\wh S f(-\eta)||\hat g(y)|\lesssim e^{-\frac{1}{16\sqrt{2d}}(4\sigma_{\min})^{\frac1s}\delta|\eta|^{\frac1s}}e^{-\eps|y|^{\frac1s}},\quad (y,\eta)\in\rdd, 
    \end{equation*}
    for every $\delta<2^{-\frac{1}{2s}}\eps.$
    As in the proof of Proposition \ref{metadecay2}, we have
    \begin{equation*}
         \displaystyle\sup_{\delta< 2^{-\frac{1}{2s}}\eps }\frac{1}{16\sqrt{2d}}(4\sigma_{\min})^{\frac1s}\delta\leq \eps,
    \end{equation*}
    hence \begin{equation*}
         |\widehat{V_g\wh S f}(\zeta)|\lesssim e^{-\frac{1}{16\sqrt{2d}}(4\sigma_{\min})^{\frac1s}\delta|\zeta|^{\frac1s}}.
    \end{equation*}
The condition $s\geq 1$ implies 
\begin{equation*}
    \frac{1}{16\sqrt{2d}}(4\sigma_{\min})^{\frac1s}=(2^{\frac1s-1})\frac{1}{8\sqrt{2d}}(2\sigma_{\min})^{\frac1s}\leq \frac{1}{8\sqrt{2d}}(2\sigma_{\min})^{\frac1s},
\end{equation*}
and we conclude that
\begin{equation*}
    |V_g\wh S f(z)|\lesssim e^{-\frac{1}{16\sqrt{2d}}(4\sigma_{\min})^{\frac1s}\delta|z|^{\frac1s}},
\end{equation*}
    that is the claim.
\end{proof}

\subsection{More on Gaussian decay --- the case $s=1/2$}\label{sec-gau}
We address here the question of how the separate decay of $f$ and $\hat f$ reflects into the joint phase-space decay of the STFT, in the vein of Theorem \ref{stftdecay}. In general, for $s\ge 1/2$ we aim to find the sharp exponents $a,a',b,b'>0$ such that, for all $x,\xi \in \rd$, the conditions
\begin{equation}\label{equivalence1}
		|f(x)|\lesssim e^{-a|x|^{1/s}},\quad 	|\hat f(\xi)|\lesssim e^{-b|\xi|^{1/s}},
	\end{equation}
are \textit{equivalent} to
\begin{equation}\label{equivalence2}
 |V_ff(x,\xi)|\lesssim e^{-(a'|x|^{1/s}+b'|\xi|^{1/s})}.
\end{equation}
The problem seems non-trivial even in the simpler Gaussian case $s=1/2$. In fact, in \cite{DGL} it is conjectured that the equivalence \eqref{equivalence1} $\Leftrightarrow$  \eqref{equivalence2} holds with $a'={a}/{2}$ and $b'={b}/{2}$. 

On the one hand, the marginal properties of the Wigner distribution in \eqref{Marginals} and its representation via STFT in \eqref{Wig-stft} are enough to show that \eqref{equivalence2} $\Rightarrow$  \eqref{equivalence1} holds with $a'=a/2,b'=b/2$ as conjectured. Sharpness follows by testing the claim on the Gaussian $\f(x)=e^{-a|x|^2}$. 
 
On the other hand, a counterexample shows that  \eqref{equivalence1} $\Rightarrow$ \eqref{equivalence2} fails with the conjectured exponents. In fact, consider
\begin{equation}\label{counter}
	f(x)=e^{i\pi |x|^2}e^{-a |x|^2}=\wh S  \f(x), \qquad x\in\bR^d,
\end{equation}
where $\f(x)=e^{-a|x|^2}$ and $\wh S $ is the metaplectic operator acting as the pointwise multiplication by the chirp $e^{i\pi |x|^2}$. By means of standard Gaussian integral formulas (see for instance \cite{folland89}), we get
\begin{align*}
	\hat f(\xi)&=\int_{\rd}e^{i\pi |x|^2}e^{-a|x|^2}e^{-2\pi i\xi x} \dd x =\int_{\rd}e^{-\pi(\frac{a}{\pi}-i)|x|^2}e^{-2\pi i\xi x} \dd x\\
	&=\Big(\frac{a}{\pi}-i\Big)^{-d/2}e^{-\frac{\pi^2}{a^2+\pi^2}|\xi|^2}e^{-i\frac{\pi}{\pi^2+a^2}|\xi|^2},
\end{align*}
so that
\begin{equation}
		|f(x)|= e^{-a|x|^{2}},\qquad 
		|\hat f(\xi)|=(a^2/\pi^2+1)^{-d/4} e^{-b|\xi|^2}, \quad b=\frac{\pi^2}{a^2+\pi^2}.
\end{equation}
However, by combining \eqref{Wig-stft} and \eqref{eq-sympcovwig}, or using \cite[Proposition 1.2.14]{Elenabook}, we obtain
\begin{equation}
	|V_ff(x,\xi)|=|V_\f\f(x,\xi-x)|=\left(\frac{2a}{\pi}\right)^{-d/2}e^{-\frac{a}{2}|x|^2}e^{-\frac{\pi^2}{2a}|\xi-x|^2}\not\leq Ce^{-\frac{a}{2}|x|^2}e^{-\frac{\pi^2}{2(a^2+\pi^2)}|\xi|^2}.
\end{equation}

To summarize, only the conjectured implication
\begin{equation*}
	|V_ff(x,\xi)|\lesssim e^{-\frac{a}{2}|x|^2-\frac{b}{2}|\xi|^2} \quad \Longrightarrow \quad |f(x)|\lesssim e^{-a|x|^{2}}\,\quad\mbox{and}\quad
		|\hat f(\xi)|\lesssim e^{-b|\xi|^2}
		\end{equation*}
happens to hold. The converse requires different exponents, as proved in the following result.

\begin{proposition}\label{prop-gg}
	Assume $a,b>0$ and let $f\in L^2(\rd)$ be such that, for all $x,\xi\in\rd$,
	\begin{align}
		\label{d-e-1}
		&|f(x)|\lesssim e^{- a|x|^2},\\
		\label{d-e-2}
		&|\hat f(\xi)|\lesssim e^{- b|\xi|^2}.
	\end{align}
	Then,
	\begin{equation}\label{STFT-decay-s12}
		|V_ff(x,\xi)|\lesssim e^{-\frac{a}{4}|x|^2-\frac{b}{4}|\xi|^2}, \qquad (x,\xi)\in\rdd.
	\end{equation}
\end{proposition}
\begin{proof}
	For $f\in L^2(\rd)$ satisfying the decay estimate \eqref{d-e-1} we have
	\begin{align}
		|V_ff(x,\xi)|&\leq\int_{\rd}|f(t)f(t-x)|\dd{t}  \lesssim\int_{\rd}e^{- a |t|^2}e^{- a|t-x|^2}\dd{t} =e^{- a|x|^2}\int_{\rd}e^{-2 a |t|^2}e^{2 a tx}\dd{t} \\
		&=\left(\frac{2a}{\pi}\right)^{-d/2}e^{- a|x|^2}e^{\frac{a}{2}|x|^2} =\left(\frac{2a}{\pi}\right)^{-d/2}e^{- \frac{a}{2}|x|^2}.
	\end{align}
	Similarly, using \eqref{d-e-2} and the fundamental identity of time-frequency analysis \eqref{ed1}, we see that
	\begin{align}
		|V_ff(x,\xi)|=|V_{\hat f}\hat f(\xi,-x)|\lesssim \left(\frac{2b}{\pi}\right)^{-d/2}e^{- \frac{b}{2}|\xi|^2}.
	\end{align}
The conclusion follows by means of Proposition \ref{prop-optbound}, or even by direct computation, since
	\begin{align}
	|V_ff(x,\xi)|^2\lesssim \left(\frac{4ab}{\pi^2}\right)^{-d/2}e^{-\frac{1}{2}(a|x|^2+b|\xi|^2)}, 
	\end{align}
	which proves \eqref{STFT-decay-s12}.
\end{proof}

For the sake of completeness, we report here analogues of Propositions \ref{metadecay1} and \ref{metadecay2} in the Gaussian case. When $s=1/2$ it is  interesting to highlight the dependence on the dimension $d$. 

\begin{proposition}\label{metadecayGaussian}
Consider $f\in\cS^{ 1/2 }_{ 1/2 ,\eps}(\rd)$ and $\wh S \in \Mp(d,\bR)$.
\begin{itemize}
    \item [(i)] If $d=1$, then we have \begin{equation*}
        \wh S f\in \begin{cases}
            \cS^{ 1/2 }_{ 1/2 ,2\sigma_{\min}^2\eps}(\bR) &  \mbox{if}\,\,\sigma_{\min}\leq \sqrt2/2\\  
            \cS^{ 1/2 }_{ 1/2 ,\eps}(\bR) & \mbox{if}\,\,\sigma_{\min}>\sqrt{2}/{2}.
        \end{cases}
    \end{equation*}
    \item[(ii)] If $d\geq 2$, then $\wh S f\in \cS^{ 1/2 }_{ 1/2 ,\frac{2}{d}\sigma_{\min}^2\eps}(\rd)$.
\end{itemize}
    
\end{proposition}

\begin{proof}
    The proof strategy mirrors the one of Propositions \ref{metadecay1} and \ref{metadecay2}. We fix $g\in\cS^{ 1/2 }_{ 1/2 ,\eps}(\rd)\setminus\{0\}$ and resort to Corollary \ref{cor0e} with $s= 1/2 $ to obtain
    \begin{equation*}
        |V_g\wh S f(z)|\lesssim e^{-\frac{\eps}{d}\sigma_{\min}^2|z|^2}=e^{- \frac12 (2\frac{\eps}{d}\sigma_{\min}^2)|z|^2},\quad \forall z\in\rdd.
    \end{equation*}
    
    If $d=1$, then \begin{equation*}
             |V_g\wh S f(z)|\lesssim e^{- 1/2 (2\eps\sigma_{\min}^2)|z|^2},\quad \forall z\in\bR^2,
        \end{equation*}
        and the claim follows by Proposition \ref{viceversaSTFT}, since $\sigma_{\min}\leq \frac{\sqrt{2}}{2}$ if and only if $2\eps\sigma_{\min}^2\leq \eps$. 
        
        If $d\geq 2$ , then $\frac{2\eps}{d}\sigma_{\min}^2\leq \eps \sigma_{\min}^2\leq \eps$ and again we get the result as a consequence of Proposition \ref{viceversaSTFT}.
\end{proof}

\begin{remark}
    Using  Proposition \ref{metadecayGaussian}  one can easily infer the related properties for $V_g f$ and $V_g \wh{S}f$, following the pattern in  the proofs of Corollaries \ref{GSVgSf1} and \ref{GSVgSf2}.
\end{remark}

\section{Applications}\label{sec:appl}
As anticipated in the Introduction, the evolution propagators of Schr\"odinger problems with quadratic Hamiltonian \eqref{SwQh} are metaplectic operators, their associated Hamiltonian flows corresponding to the symplectic projections of these operators. The purpose of this section is twofold. First, we compute Euler decompositions for the Hamiltonian flows arising from quantum dynamics --- including the free particle (also in presence of uniform magnetic field) and the anisotropic harmonic oscillator; second, we compute Euler decompositions of the generators of $\Sp(d,\bR)$.

Our focus is on the spreading matrix $D'U$, appearing in the decay estimate \eqref{spreading} of Theorem \ref{teoe00}.
For the benefit of the reader, Algorithm~\ref{Agorithm1} summarizes the procedure used to compute the Euler decompositions. Since the full computation can be quite involved, we omit the detailed steps and provide only the matrices $D$ and $U$ from which $D'U$ can be easily inferred.

\begin{algorithm}
\caption{\label{Agorithm1} Euler decomposition of a symplectic matrix}
\begin{algorithmic}[1]
\Require A matrix $S \in \mathrm{Sp}(d, \mathbb{R})$.
\Ensure A triple $(U, V, \Sigma)$ such that $S = U^\top D V$, where:
\begin{itemize}
    \item $D = \Sigma \oplus \Sigma^{-1}$ with $\Sigma = \mathrm{diag}(\sigma_1, \ldots, \sigma_d)$,
    \item $\sigma_1 \geq \cdots \geq \sigma_d \geq 1$ are the $d$ largest singular values of $S$,
    \item $U$ and $V$ are orthogonal and symplectic matrices.
\end{itemize}
\State \textbf{Compute} $SS^\top$, and determine the singular values $\sigma_1, \ldots, \sigma_d, \sigma_1^{-1}, \ldots, \sigma_d^{-1}$ of $S$. Define the diagonal matrices $\Sigma$ and $D$ as above.
\State \textbf{Diagonalize} $SS^\top = \tilde{U}^\top D^2 \tilde{U}$.
\State \textbf{Modify} $\tilde{U}$ to obtain a symplectic matrix $U$. In particular, when $\tilde{U}$ has diagonal blocks (as in our applications), this can be achieved by changing the sign of some of its rows.
\State \textbf{Define} $V = D^{-1} U S$.
\end{algorithmic}
\end{algorithm}

\subsection{Free particle}\label{sec free prop}
The simplest example is the free particle,
	\begin{equation}\label{cp}
		\begin{cases}
			\displaystyle i\frac{1}{2\pi}\frac{\partial u}{\partial t} (t,x)=-\frac{\Delta}{4\pi^2} u,\\
			u(0,x)=u_0(x),
		\end{cases}
	\end{equation}
	with $(t,x)\in\bR\times\bR^d$, $d\geq1$. The solution is given by
\[
u(t,x) = \widehat{S}_t u_0(x) = \int_{\mathbb{R}^d} e^{2\pi i \left( x \cdot \xi - t \, |\xi|^2 \right)} \, \widehat{u_0}(\xi) \, \mathrm{d}\xi, \quad t \in \mathbb{R}.
\]
The free particle propagator is given by $\widehat{S}_t(t)=e^{it\Delta/2\pi}$ and the corresponding classical flow is
\begin{equation}\label{schr prop}S_t= \begin{pmatrix}
		I & 2t I \\ O & I
	\end{pmatrix}, \quad t \in \bR. \end{equation} 
Recalling the results in \cite[Section 4]{spreading}, a straightforward computation shows that the largest $d$ singular values  of $S_t$ coincide: \[ \sigma_j = \sigma(t) = (1+2t^2 + 2(t^2+t^4)^{1/2})^{1/2} = \sqrt{1+t^2}+|t|, \quad j=1,\ldots,d. \] 
In particular,  $\sigma(t)\asymp 1+|t|$, $t \in \bR$. An example of Euler decomposition  $(U_t,V_t,\Sigma_t)$ of $S_t$ for $t\ge0$ is given by 
\[ U_t = (1+\sigma(t)^2)^{-1/2} \begin{pmatrix}
	\sigma(t)I & I \\ -I & \sigma(t)I
\end{pmatrix}, \quad V_t = (1+\sigma(t)^2)^{-1/2} \begin{pmatrix} I & \sigma(t)I \\ -\sigma(t)I & I \end{pmatrix}. \] 

The spreading phenomenon manifests itself as a dilation by 
\begin{equation}\label{spreadingFP}
    D'_tU_t = (1+\sigma(t)^2)^{-1/2}\begin{pmatrix} I & \sigma(t)^{-1}I \\ -I & \sigma(t)I \end{pmatrix}.
\end{equation} 
The spreading phenomenon of the free particle is conceptually outlined in Figure \ref{fig:1}, where we illustrate the effect of $D'_t U_t$ with $d=1$ on the unit disk centered at the origin in $\bR^2$ --- as a simple toy model for the essential time-frequency support of a Gabor wave packet.

\begin{center}
    \begin{figure}
        \centering
        \includegraphics[width=0.9\linewidth]{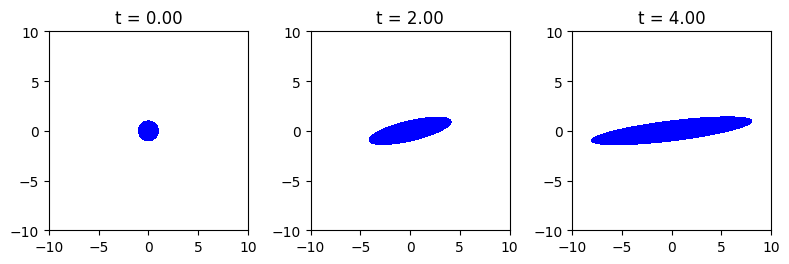}
        \caption{\label{fig:1}\textbf{Spreading phenomenon for the free particle.} The image of the unit ball of the phase space ($t = 0$) under the linear transformation $D'_t U_t$ in \eqref{spreadingFP}. As $t$ increases, the map produces both rotation and anisotropic stretching, with the spreading exhibiting a pronounced directional component.}
    \end{figure}
\end{center}
\subsection{Harmonic oscillator}\label{sec HS}

The classical harmonic oscillator has been analyzed by Knutsen~\cite[Example 4.2]{Knutsen}, where the Hamiltonian takes the form
\[
H(z) = \frac{1}{2m} \sum_{j=1}^d \xi_j^2 + \frac{m}{2} \sum_{j=1}^d \omega_j^2 x_j^2,
\quad \text{with } z = (x, \xi), \,m>0, \,\omega_j\in\bR.
\]
This leads to the following Schr\"odinger equation:
\begin{equation} \label{eq:sho-schrodinger}
i\frac{1}{2\pi} \frac{\partial u}{\partial t}(t,x) = \left( - \frac{1}{8\pi^2 m} \Delta + \frac{m}{2} \sum_{j=1}^d \omega_j^2 x_j^2 \right) u(t,x).
\end{equation}

An explicit expression for the symplectic matrix $S_t$, having $d\times d$ block decomposition, is given by:
\begin{equation}\label{blockSintro}
S_t=\begin{pmatrix}\cA_t & \cB_t\\ \cC_t & \cD_t\end{pmatrix}\in\bR^{2d\times2d}, \qquad \cA_t,\cB_t,\cC_t,\cD_t\in\bR^{d\times d},
\end{equation}
associated with the system is provided in~\cite[Example 4.2]{Knutsen}.  Namely, the  blocks are given by:
\[
\cB_t = \frac{1}{m} \, \mathrm{diag} \left( \frac{\sin(\omega_j t)}{\omega_j} \right), \quad
\cA_t = \cD_t = \mathrm{diag} \left( \cos(\omega_j t) \right), \quad
\cC_t = -m \, \mathrm{diag} \left( \omega_j \sin(\omega_j t) \right).
\]
To compute an Euler decomposition of $S_t$, let us consider the case where $d=1$ first. Following Algorithm~\ref{Agorithm1}, we compute:
\begin{align}
S_t S_t^\top &=\begin{pmatrix}
    \cos^2(\omega t)+\dfrac{1}{m^2\omega^2}\sin^2(\omega t) & \left(\dfrac{1}{m\omega}-m\omega\right)\sin(\omega t)\cos(\omega t)\\
    \left(\dfrac{1}{m\omega}-m\omega\right)\sin(\omega t)\cos(\omega t) & 
     \cos^2(\omega t){m^2\omega^2}\sin^2(\omega t)
\end{pmatrix} \eqqcolon \begin{pmatrix}
        \cX_t & \mathcal{Y}_t\\
        \mathcal{Y}_t & \cZ_t
    \end{pmatrix}
\end{align}
We need to distinguish four scenarios.
\subsubsection*{First scenario.} If $m\omega=\pm1$ or $\omega=0$, then
\begin{equation}\label{fractionalFT}
    S_t=\begin{pmatrix}
        \cos(\omega t) & -\sgn(m\omega)\sin(\omega t)\\
        \sgn(m\omega) \sin(\omega t) & \cos(\omega t)
    \end{pmatrix},
\end{equation}
which is symplectic and orthogonal. In this case, $D_t=I$, $U_t=I$ and $V_t=S_t^\top $. 

\subsubsection*{Second scenario.} If $\omega\neq0$ and $\omega t=\frac{\pi}{2}+2k\pi$, then $\cos(\omega t)=0$ and $\sin(\omega t)=(-1)^k$, so that 
\begin{equation}
    S_t=\begin{pmatrix}
        0 & \dfrac{(-1)^k}{m\omega}\\
        m\omega(-1)^{k+1} & 0
    \end{pmatrix}, \qquad S_tS_t^\top =
    \begin{pmatrix}
        \dfrac{1}{m^2\omega^2} & 0\\
        0 & m^2\omega^2
    \end{pmatrix}.
\end{equation}
In this case, the singular values of $S_t$ are $\sigma_1=|m\omega|$ and $\sigma_2=\frac{1}{|m\omega|}$, and $S_t=U_t^\top D_tV_t$ with 
\begin{equation}
   U_t=I, \quad V_t=\begin{pmatrix}
        0 & (-1)^{k}\sgn(m\omega)\\
        (-1)^{k+1}\sgn(m\omega) & 0
    \end{pmatrix}, \quad D_t=\begin{pmatrix}
        \dfrac{1}{|m\omega|} & 0\\
        0 & {|m\omega|}
    \end{pmatrix}
\end{equation}
if $|m\omega|\leq 1$, or
\begin{equation}
    V_t=(-1)^{k+1}\sgn(m\omega)I, \quad D_t=\begin{pmatrix}
        {|m\omega|} & 0\\
        0 & \dfrac{1}{|m\omega|}
    \end{pmatrix},\quad U_t=\begin{pmatrix}
        0 & 1\\
        -1 & 0
    \end{pmatrix}
\end{equation}
otherwise.

\subsubsection*{Third scenario.} If $\omega t=k\pi$, $k\in\bZ$, then $S_t=(-1)^kI$, and $U_t=S_t$, $D_t=I$ and $V_t=I$.

\subsubsection*{Fourth scenario.} If $m\omega\neq\pm1$ and $\omega\neq0$, and $\cos(\omega t),\sin(\omega t)\neq 0$, the eigenvalues of $S_t^\top S_t$ can be easily computed as
\begin{equation}\label{lambdaOA}
   \lambda_\pm(t) = 1+\beta(t)^2\pm\beta(t)\sqrt{2+\beta(t)^2},
\end{equation}
where
\begin{equation}
\beta(t)=\left|\frac{1}{\sqrt{2}}\left(\frac{1}{m\omega}-m\omega\right)\sin(\omega t)\right|.
\end{equation}
The singular values are $\sigma(t)$ and $\sigma(t)^{-1}$, where
\begin{equation}\label{singValHO}
    \sigma(t)=\sqrt{1+\frac{\beta(t)^2}{2}}+\frac{\beta(t)}{\sqrt{2}},
\end{equation}
Recall that $\cX_t$ and $\mathcal{Y}_t$ are the two upper blocks of $S_t S_t^\top $. 
Then, $S_t=U_t^\top D_tV_t$ is an Euler decomposition of $S_t$, with
\begin{equation}\label{defVtpm}
    U_t=\begin{pmatrix}
        \dfrac{\mathcal{Y}_t}{\sqrt{\mathcal{Y}_t^2+(\lambda_+(t)-\cX_t)^2}} & \dfrac{\lambda_+(t)-\cX_t}{\sqrt{\mathcal{Y}_t^2+(\lambda_+(t)-\cX_t)^2}}\\
        -\dfrac{|\mathcal{Y}_t|}{\sqrt{\mathcal{Y}_t^2+(\lambda_-(t)-\cX_t)^2}} & -\sgn(\mathcal{Y}_t)\dfrac{\lambda_-(t)-\cX_t}{\sqrt{\mathcal{Y}_t^2+(\lambda_-(t)-\cX_t)^2}}
    \end{pmatrix},
\end{equation}
and $V_t$ computed accordingly as $V_t=D_t^{-1}U_tS_t$. The correcting factor $-\sgn(\mathcal{Y}_t)$ is added so that $\det(U_t)>0$, and hence $\det(U_t)=1$, guaranteeing that $U_t\in\Sp(1,\bR)$.

The general case $d>1$ can be approached with a standard permutation argument. The core idea is rearranging the rows and the columns of $S_t$ using a permutation $P$ so that $PS_tP^\top = S_t^{(1)}\oplus\ldots\oplus S_t^{(d)}$, where
\begin{equation}
    S_{t}^{(j)}
    =\begin{pmatrix}
        \cos(\omega_jt) & \dfrac{1}{m\omega_j}\sin(\omega_j t)\\
        -m\omega_j\sin(\omega_jt) & \cos(\omega_jt)
    \end{pmatrix}, \qquad j=1,\ldots,d.
\end{equation}
Then, if $S_t^{(j)} = (U_t^{(j)})^\top D_t^{(j)} V_t^{(j)}$ is the Euler decomposition of $S_t^{(j)}$ for $j = 1, \ldots, d$, the Euler decomposition of $S_t$ is obtained by conjugating each factor with respect to the inverse permutation. Specifically, for each $j = 1, \ldots, d$, we compute the singular values $\sigma^{(j)}(t)$ and $\sigma^{(j)}(t)^{-1}$ of $S_t^{(j)}$ as defined in \eqref{singValHO}. Next, we consider a permutation matrix $P\in\bR^{2d\times 2d}$ such that:
\begin{enumerate}[(1)]
    \item $PS_tP^\top =S_t^{(1)}\oplus\ldots\oplus S_t^{(d)}$.
    \item $\sigma^{(1)}(t)\geq \ldots\geq \sigma^{(d)}(t)\geq 1$.
\end{enumerate}
Then, $S_t=U_t^\top D_tV_t$ is a SVD of $S_t$, where 
\begin{align} 
    &U_t=P^\top (U_t^{(1)}\oplus\ldots\oplus U_t^{(d)}) P, \\
    &D_t=P^\top (D_t^{(1)}\oplus\ldots \oplus D_t^{(d)})P, \\
    &V_t=P^\top (V_t^{(1)} \oplus \ldots \oplus V_t^{(d)})P.
\end{align}
Here, the matrices $V_t^{(j)}$ and $U_t^{(j)}$ ($j=1,\ldots,d$) are computed using the argument in dimension 1.
Let us show that $U_t,V_t\in \Sp(d,\bR)$. We denote with
\begin{equation}
    J^{(j)}=\begin{pmatrix}
        0 & 1\\
        -1 & 0
    \end{pmatrix}, \qquad j=1,\ldots,d,
\end{equation}
so that, by definition of $P$, we have $PJP^\top =J^{(1)}\oplus\ldots\oplus J^{(d)}$. Consequently,
\begin{align}
    U_t^\top  J U_t &= P^\top ((U_t^{(1)})^\top \oplus\ldots\oplus(U_t^{(d)})^\top )PJP^\top (U_t^{(1)}\oplus\ldots \oplus U_t^{(d)})P\\
    &= P^\top ((U_t^{(1)})^\top  JU_t^{(1)} \oplus\ldots\oplus (U_t^{(d)})^\top JU_t^{(d)})P\\
    &=P^\top (J^{(1)}\oplus\ldots\oplus J^{(d)})P=J.
\end{align}
Similarly, $V_t^\top JV_t=J$. Hence, $U_t,V_t\in\Sp(d,\bR)$, proving that $S_t=U_t^\top D_tV_t$ is an Euler decomposition of $S_t$. The spreading transformation \( D'_t U_t \) is illustrated in Figure~\ref{fig:2} by means of a representative example. The matrix \( D'_t U_t \) is \( \pi/\omega \)--periodic in time and acts as a composition of a rotation and an anisotropic dilation.
\begin{figure}
    \centering
    \includegraphics[width=0.9\linewidth]{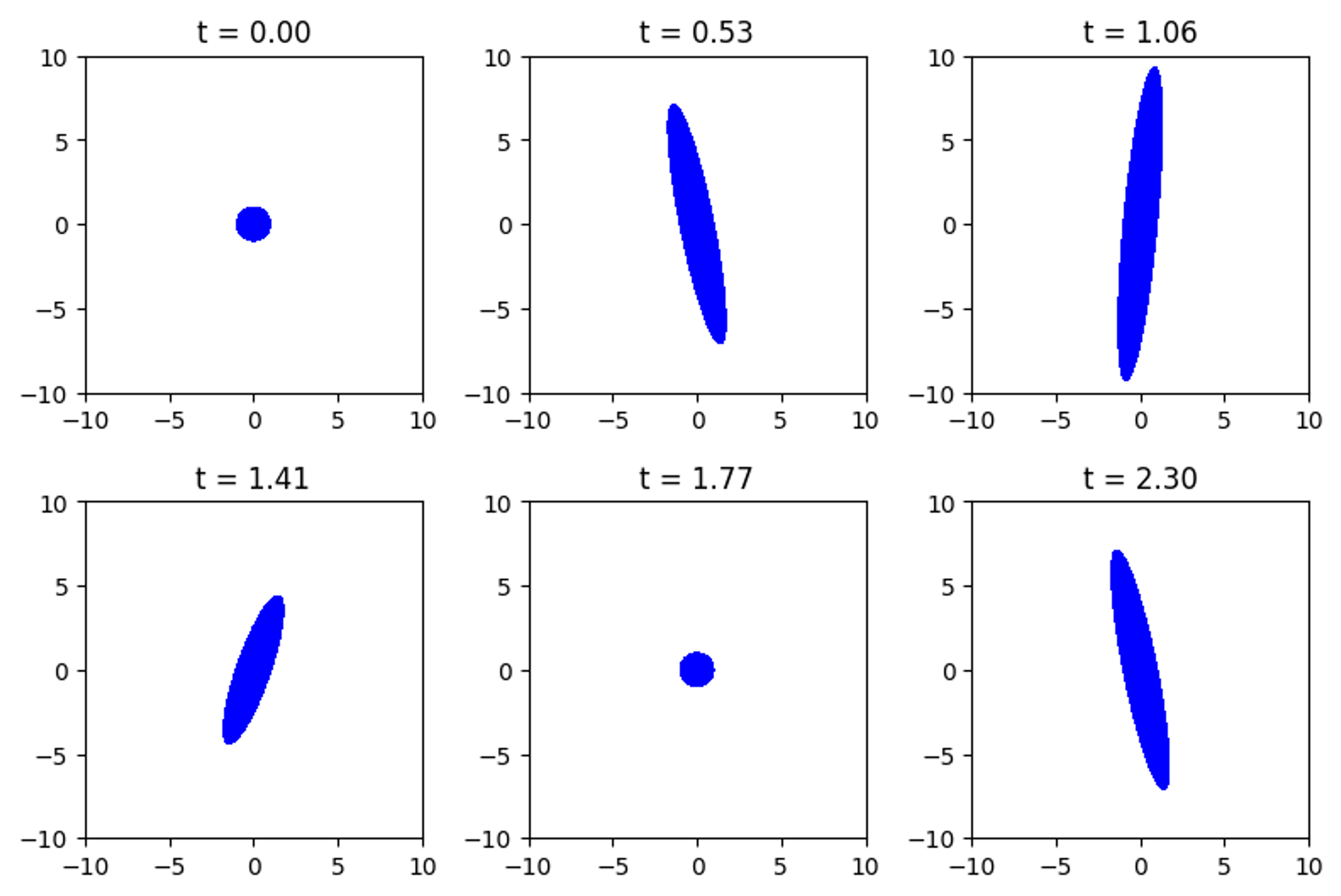}
    \caption{\label{fig:2}\textbf{Spreading phenomenon for the harmonic oscillator.} The image of the unit ball of the phase space ($t = 0$) under the linear transformation $D'_t U_t$, with $d=1$. The parameters $m$ and $\omega$ have been chosen randomly (here, $\omega=1.7762$ and $m = 1.7939$). The times have been sampled as $\alpha\pi/\omega$ with $\alpha\in\{0,0.3,0.6, 0.8,1,1.3\}$. Observe the periodicity of the roto-dilation. In particular, $D'_tU_t$ acts as the identity at time $t=\pi/\omega\approx1.77$.}
\end{figure}

\subsection{Uniform magnetic potential}

Let us consider the following Hamiltonian: 
\begin{equation}\label{UMP}
    H(x,\xi)= \frac{1}{2m}\sum_{j=1}^d \xi_j^2 + \omega(Bx\cdot \xi-B\xi\cdot x)+\frac{m\omega^2}{2}\sum_{j=1}^d x_j^2,
\end{equation}
where $m>0$, $\omega\in\bR\setminus\{0\}$ and $B\in \bR^{d\times d}$ is a non-trivial skew-symmetric orthogonal matrix --- namely, satisfying $B^\top =-B$ and $B^\top B=I$. Note that the assumptions on $B$ force the dimension $d$ to be even. Indeed, for $d=2$ and $B= J$, the quantization of the so-called \textit{Landau Hamiltonian} in \eqref{UMP} represents a simplified model for the quantum dynamics of a charged particle in the plane under the action of a uniform, orthogonal magnetic field. For more information on the spectral properties of this and related problems (e.g., the twisted Laplacian) see \cite{goslue,gos_landau,trapasso_twist} and the references therein.  

The canonical transformation associated with \eqref{UMP} is
\begin{equation}\begin{split}
	S_t =\begin{pmatrix}
    \cos^2(\omega t)I+\sin(\omega t)\cos(\omega t)B & \dfrac{1}{m\omega}(\sin^2(\omega t)B+\sin(\omega t)\cos(\omega t)I)\\
    -m\omega(\sin^2(\omega t)B+\sin(\omega t)\cos(\omega t) I) & \cos^2(\omega t)I+\sin(\omega t)\cos(\omega t)B
\end{pmatrix},
\end{split}
\end{equation}
obtained by applying the half-angle formulae to \cite[Section 7]{CGRSchrRep}.

Using that $B^\top =-B$ and $B^\top B=I_d$, a straightforward computation shows that
\begin{equation}\begin{split}\label{StTSt}
	S_t S_t^\top &= \begin{pmatrix}
    \left[\cos^2(\omega t)+\dfrac{1}{m^2\omega^2}\sin^2(\omega t)\right]I & \left(\dfrac{1}{m\omega}-m\omega\right)\sin(\omega t)\cos(\omega t) I\\
    \left(\dfrac{1}{m\omega}-m\omega\right)\sin(\omega t)\cos(\omega t)I & \left[ \cos^2(\omega t)+{m^2\omega^2}\sin^2(\omega t) \right]I
    \end{pmatrix}\\
    &\eqqcolon\begin{pmatrix}
	\mathcal{X}_t I& \mathcal{Y}_t I\\
	 \mathcal{Y}_t I &\mathcal{Z}_t I
 \end{pmatrix}.
\end{split}
\end{equation}
We distinguish three scenarios.
\subsubsection*{First scenario} If $m\omega=\pm1$, then $S_t S_t^\top=I$ and, therefore $\sigma_1=\ldots=\sigma_{2d}=1$. In this simple case, $D'_t=U_t=I$.

\subsubsection*{Second scenario} If $m\omega\neq\pm1$, and $t=\frac{k\pi}{2\omega}$ ($k\in\bZ$), we have 
\begin{equation}
	S_t S_t^\top=\frac{1}{2}\begin{pmatrix}
		\left[\left(1+\dfrac{1}{m^2\omega^2}\right)+\left(1-\dfrac{1}{m^2\omega^2}\right)(-1)^k\right]I & O\\
		O & \left[\left(1+{m^2\omega^2}\right)+\left(1-{m^2\omega^2}\right)(-1)^k\right]I
	\end{pmatrix},
\end{equation}
and we distinguish two sub-cases. If $k$ is even, $S_t S_t^\top=I$ and, again we have $\sigma_1=\ldots=\sigma_{2d}=1$. 
If $k$ is odd,
\begin{equation}
	S_t S_t^\top=\begin{pmatrix}
		\dfrac{1}{m^2\omega^2} I & O\\
		O & {m^2\omega^2}I
	\end{pmatrix},
\end{equation}
so that $\sigma_1=\ldots=\sigma_d=\max\{|m\omega|,|m\omega|^{-1}\} \eqqcolon \sigma(t)$. In this case, it is straightforward to decompose $S_t=U_t^\top D_tV_t$ with
\begin{equation}
	U_t=I, \quad V_t=\begin{pmatrix} O & \sgn(m\omega) B\\
	-\sgn(m\omega)B & O \end{pmatrix}
\end{equation}
if $|m\omega|\leq 1$, and
\begin{equation}
	U_t=J,\qquad 
	V_t=-\begin{pmatrix}
		\sgn(m\omega)B & O\\
		O & \sgn(m\omega)B
	\end{pmatrix}
\end{equation}
otherwise. 

\subsubsection*{Third scenario} If $m\omega\neq\pm1$, and $t\neq\frac{k\pi}{2\omega}$, we have that $(\frac{1}{m\omega}-\omega)\sin(2\omega t)\neq0$ and, consequently, the upper-right block of \eqref{StTSt} is invertible. Using Schur's complements for the computation of $\det(S_t S_t^\top-\lambda I_{2d})$, we infer that the eigenvalues $\lambda_\pm(t)$ of $S_t S_t^\top$ have the same expression in \eqref{lambdaOA} and, consequently, the largest singular values of $S_t$, $\sigma(t)$, are the same as in \eqref{singValHO}. Accordingly, the other singular value is $\sigma(t)^{-1}$, and they both come with multiplicity $d$. 
Consequently, $D_t=\sigma(t)I_d\oplus \sigma(t)^{-1}I_d$, and 
\begin{equation}
    U_t=\begin{pmatrix}
        \dfrac{\mathcal{Y}_t}{\sqrt{\mathcal{Y}_t^2+(\lambda_+(t)-\cX_t)^2}}I & \dfrac{\lambda_+(t)-\cX_t}{\sqrt{\mathcal{Y}_t^2+(\lambda_+(t)-\cX_t)^2}}I\\
        -\dfrac{|\mathcal{Y}_t|}{\sqrt{\mathcal{Y}_t^2+(\lambda_-(t)-\cX_t)^2}}I & -\sgn(\mathcal{Y}_t)\dfrac{\lambda_-(t)-\cX_t}{\sqrt{\mathcal{Y}_t^2+(\lambda_-(t)-\cX_t)^2}}I
    \end{pmatrix},
\end{equation}
where now the blocks $\cX_t$ and $\mathcal{Y}_t$ are defined as in \eqref{StTSt}. The factor $-\sgn(\mathcal{Y}_t)$ in the lower blocks makes $U_t$ symplectic. The matrix $V_t$ is then computed using the relation $S_t=U_t^\top D_tV_t$. Although the existence of a non-zero $d \times d$ orthogonal matrix $B$ satisfying $B^\top=-B$ requires $d$ to be even, the conceptual illustration in Figure~\ref{fig:2}, which depicts the spreading phenomenon for the harmonic oscillator, remains relevant in the present setting of the uniform magnetic potential, due to the block-diagonal structure of the matrices $U_t$ and $D_t$.

\subsection{Euler decompositions of the generators of $\Sp(d,\bR)$}\label{subsec:metaplecticGens}
   Finally, we compute the Euler decompositions of the generators of the symplectic group, that are $J$, defined in \eqref{defJ}, and matrices in the form
\begin{equation} \label{defDLVC}
	\mathcal{D}_E := \begin{pmatrix}
		E^{-1} & O \\
		O & E^\top
	\end{pmatrix}, \qquad
	V_Q := \begin{pmatrix}
		I & O \\
		Q & I
	\end{pmatrix},
\end{equation}
    for $d\times d$ matrices $Q=Q^\top$ and $E\in\GL(d,\bR)$.
    
    \subsubsection{Fourier transforms.} As far as $J$ is concerned, the singular values of $J$ are trivially $\sigma_1=\ldots=\sigma_{2d}=1$, and the corresponding Euler decomposition is $J=U^\top DV$, where $D=I$, $U=J^\top$ and $V=I$. Let us comment about the so-called {\em partial Fourier transforms}. For a subset of indices $\cJ=\{1\leq j_1<\ldots<j_r\leq d\}\subseteq\{1,\ldots,d\}$, consider the diagonal matrix $I_\cJ$ with diagonal entries $I_{jj}=1$ if $j\in\cJ$ and $0$ otherwise. For $x\in\rd$, set $x_\cJ =I_\cJ x$ and $x_{\cJ^c}=x-x_\cJ$. The metaplectic operator with projection
    \begin{equation}
        \Pi_\cJ=\begin{pmatrix}
            I-I_\cJ & I_\cJ\\
            -I_\cJ & I-I_\cJ
        \end{pmatrix}
    \end{equation}
    is, up to a phase factor, a so-called fractional Fourier transform, 
    \begin{equation}
        \cF_\cJ f(x)=\int_{\bR^r}f(x_{\cJ^c}+y_\cJ)e^{-2\pi ix_\cJ\cdot  y_\cJ}dy_\cJ, \qquad f\in\cS(\rd),
    \end{equation}
    where $dy_\cJ =dy_{j_1}\ldots dy_{j_r}$. Since $\Pi_\cJ^\top\Pi_\cJ=I$, $\sigma=1$ is a singular value of $\Pi_\cJ$ with multiplicity $2d$. Again, if $V=I$, $D=I$ and $U=\Pi_\cJ^\top$, we have $\Pi_\cJ =U^\top DV$. 
    
    Now, let $\Gamma=\mbox{span}\{v_1,\ldots,v_r\}\subseteq\rd$ be a subspace of $\rd$ of dimension $r$, with orthonormal basis $\{v_1,\ldots,v_r\}$. Let us decompose $x=x_1+x_2\in\rd$, with $x_1\in\Gamma$ and $x_2\in\Gamma^\perp$. Consider a rotation $R$ so that $Rv_j=e_j$ for every $j=1,\ldots,r$, where $\{e_1,\ldots,e_r$\} are the first $r$ vectors of the canonical basis of $\bR^d$. Let $\cJ=\{1,\ldots,r\}$, the action of the metaplectic operator with projection
    \begin{equation}
        \Pi_R=\begin{pmatrix}
            R^\top(I-I_\cJ)R & R^\top I_\cJ R\\
            -R^\top I_\cJ R & R^\top(I-I_\cJ)R
        \end{pmatrix}=\cD_{R}\Pi_\cJ\cD_{R^\top}.
    \end{equation}
    is given by 
    \begin{align}
        \cF_\Gamma f(\xi_1+x_2)&=\cF_{\cJ}(f\circ R^\top)(R(\xi_1+x_2))\\
        &=\int_{\bR^r}f\circ R^\top(y_1^1,\ldots,y_1^r,Rx_2)e^{-2\pi iR\xi_1\cdot y_1}dy_1,
    \end{align}
    where $\xi_1\in\Gamma$, $x_2\in\Gamma^\perp$ and we have identified 
    \begin{equation}
        Rx_2=(\underbrace{0,\ldots,0}_{\text{$r$ times}},y_2^1,\ldots,y_2^{d-r})\in\rd\leftrightarrow (y_2^{1},\ldots,y_2^{d-r})\in\bR^{d-r},
    \end{equation}
    and similarly for $R\xi_1$. The change of variables $R^\top y_1=x_1$ yields
    \begin{align}
         \cF_\Gamma f(\xi_1+x_2)=\int_\Gamma f(x_1+x_2)e^{-2\pi i\xi_1\cdot x_1}dx_1,
    \end{align}
    the {\em partial Fourier transform along $\Gamma$}. Then, the Euler decomposition of $\Pi_R$ is $\Pi_R=U^\top DV$, where $U=\Pi_\cJ^\top \cD_{R^\top}$, $D=I$ and $V=\cD_{R^\top}$.
    
    \subsubsection{The matrices $\cD_E$}Concerning the matrices $\cD_E$, let $E=U^\top DV$ be an SVD of $E$, then an Euler decomposition of $\cD_E$ is
    \begin{equation}
        \cD_E=\begin{pmatrix}
            V^\top & O\\
            O & V^\top
        \end{pmatrix}
        \begin{pmatrix}
            D & O\\
            O & D^{-1}
        \end{pmatrix}
        \begin{pmatrix}
            U & O\\
            O & U
        \end{pmatrix}.
    \end{equation}
    
    \subsubsection{The matrices $V_Q$} The Euler decomposition of lower triangular matrices $V_Q$ is slightly more involved. We proceed step by step. For $d=1$ and $\mu\in\bR$, let
    \begin{equation}
        V_\mu=\begin{pmatrix}
            1 & 0\\
            \mu & 1
        \end{pmatrix}.
    \end{equation}
    If $\mu=0$, then $ V_\mu=I$ and there is nothing to say. Otherwise, the eigenvalues of $V_\mu V_\mu^\top$ are
    \begin{equation}
        \lambda_\pm = 1+\frac{\mu^2}{2}\pm\sqrt{\frac{\mu^4}{4}+\mu^2}.
    \end{equation}
    and, consequently, the singular values of $V_\mu$ are
    \begin{equation}\label{singVVQ}
        \sigma_\pm=\sqrt{1+\frac{\mu^2}{4}}\pm \left|\frac{\mu}{2}\right|.
    \end{equation}
    Observe that the expression \eqref{singVVQ} also encodes the case where $\mu=0$, where the only singular value is $\sigma=1$ with multiplicity 2. Let
        \begin{equation}
            \rho_\pm=\sqrt{2\mu^2+\dfrac{\mu^4}{4}}\pm\dfrac{\mu^2}{2}
        \end{equation}
       be the norms of the eigenvectors of $V_\mu V_\mu^\top$ associated to $\lambda_\pm$ respectively. Then, the matrix
        \begin{equation}
            U=\begin{pmatrix}
                \dfrac{\mu}{\rho_+} & \dfrac{\lambda_+-1}{\rho_+}\\
                -\dfrac{|\mu|}{\rho_-} & \dfrac{-\sgn(\mu)(\lambda_--1)}{\rho_-}
            \end{pmatrix}
        \end{equation}
        is orthogonal and has $\det(U)=1$ by construction. 
        Consequently, $V_\mu=U^\top DV$ is an Euler decomposition of $V_\mu$ when $\mu\neq0$, where $U$ is defined as above, $D=\diag(\sigma_+,\sigma_-)$, and $V$ is computed accordingly.
        
        For the case $d>1$, we first consider $Q=\Delta=\diag(\lambda_1,\ldots,\lambda_d)$. Then,
        there exists a permutation $P$ so that 
        \begin{equation}\label{DecompVQ}
            PV_\Delta P^\top=V_{\lambda_1}\oplus\ldots\oplus V_{\lambda_d}, \qquad V_{\lambda_j}=\begin{pmatrix}
                1 & 0\\
                \lambda_j & 1
            \end{pmatrix},
        \end{equation}
        where, with an abuse of notation, the permutation is chosen so that $|\lambda_1|\geq\ldots\geq |\lambda_d|$. Using \eqref{DecompVQ}, we may compute the Euler decomposition of each $V_{\lambda_j}$, say $V_{\lambda_j}=U_j^\top D_jV_j$, according to the argument for $d=1$, and define
        \begin{align}
            V_\Delta&=P^\top(V_{\lambda_1}\oplus\ldots\oplus V_{\lambda_d})P\\
            &=P^\top(U_1^\top\oplus \ldots\oplus U_d^\top)(D_1\oplus \ldots\oplus D_d)(V_1\oplus\ldots\oplus V_d)P\\
            &=U^\top DV,
        \end{align}
        where 
        \begin{align}
            &U=P^\top (U_1\oplus\ldots\oplus U_d)P,\\
            &D=P^\top (D_1\oplus\ldots\oplus D_d)P,\\
            &V=P^\top (V_1\oplus\ldots\oplus V_d)P.
        \end{align}
        Observe that all these matrices are symplectic, since $PJP^\top=(J_1\oplus\ldots\oplus J_d)$, where $J_1=\ldots=J_d=J\in \Sp(1,\bR)$.
        The singular values of $V_\Delta$ are computed as in \eqref{singVVQ}, namely
        \begin{equation}
             \sigma_\pm^{(j)}=\sqrt{1+\frac{\lambda_j^2}{4}}\pm \left|\frac{\lambda_j}{2}\right|.
        \end{equation}
        Observe that by choosing $P$ so that the eigenvalues of $V_Q$ have been sorted with decreasing absolute values, the singular values of $V_Q$ are arranged as 
        \begin{equation}
            \begin{array}{ll}
            \sigma_j=\sigma^{(j)}_+, & j=1,\ldots,d,\\
            \sigma_{d+j}=\sigma_-^{(d-j+1)}, & j=1,\ldots,d.
            \end{array}
        \end{equation}
        Finally, it remains to study the general case where $Q$ is real and symmetric. Let $\Theta$  orthogonal be such that $Q=\Theta^\top\Delta \Theta$, where $\Delta=\diag(\lambda_j)$ is the diagonal matrix of the eigenvalues of $Q$. A simple computation shows that the singular values of $V_Q$ coincide with the singular values of $V_\Delta$, and
        \begin{equation}\label{equationSVDVQ2}
            V_Q=\underbrace{\begin{pmatrix}
                \Theta^\top & O \\
                O & \Theta^\top
            \end{pmatrix}\mathcal{U}^\top}_{\eqqcolon U^\top}D\underbrace{\mathcal{V}\begin{pmatrix}
                \Theta & O \\
                O & \Theta
            \end{pmatrix}}_{ \eqqcolon V}=U^\top DV,
        \end{equation}
        is an Euler decomposition of $V_Q$, where $V_{\Delta}=\mathcal{U}^\top D\mathcal{V}$ is the Euler decomposition of $V_\Delta$ computed as above. 
    
    \begin{remark}
       The same argument yields the Euler decomposition of $V_Q^\top$, $Q=Q^\top$, the symplectic projections of metaplectic Fourier multipliers
       \begin{equation}
        \mathfrak{m}_Qf=\cF^{-1}(\Phi_{-Q}\hat f), \qquad f\in L^2(\rd),
       \end{equation}
      where $\Phi_{-Q}(x)=e^{-i\pi Qx\cdot x}$.
    \end{remark}

\section*{Acknowledgments}

The authors thank the \textit{Erwin Schr\"odinger International Institute for Mathematics and Physics (ESI)},  University of Vienna. This work began during the ESI stay of the first three authors from May 5 to 9, 2025.

The authors are members of Gruppo Nazionale per l’Analisi Matematica, la Probabilit\`a e le loro Applicazioni (GNAMPA) --- Istituto Nazionale di Alta Matematica (INdAM). The present research has been developed as part of the activities of the GNAMPA-INdAM project ``Analisi spettrale, armonica e stocastica in presenza di potenziali magnetici'', award number (CUP): E5324001950001. 

\bibliographystyle{abbrv}

\end{document}